\input ./style/arxiv-general.cfg
\documentclass[aop,MSNbibl,seceqn,dvips]{arximspdf}
\makeatletter
   \@ifpackageloaded{graphicx}{}{\usepackage{graphicx}}
\makeatother
\usepackage{mathbh}

\doi{10.1214/14-AOP932} 
\volume{43}
\issue{5}
\pubyear{2015}
\firstpage{2250}
\lastpage{2281}

\makeatletter
\newcommand{\rrvert}{\vert}
\newcommand{\llvert}{\vert}

\newcommand{\eqref}[1]{(\ref{#1})}
\newtheorem{theorem}{Theorem}[section]
\newtheorem{corollary}[theorem]{Corollary}
\newtheorem{lemma}[theorem]{Lemma}
\newtheorem{proposition}[theorem]{Proposition}
\newtheorem{pb}{Problem}

\newproclaim{definition}{Definition}
\newproclaim{rem}{Remark}
\newproclaim{exa}{Example}

\newproclaim{remark}{Remark}
\newcommand{\Z}{\mathbb{Z}}
\newcommand{\E}{\mathbb{E}}
\newcommand{\R}{\mathbb{R}}
\newcommand{\bbP}{\mathbb{P}}
\newcommand{\mN}{\mathcal{N}}
\newcommand{\om}{\omega}
\newcommand{\vk}{\varkappa}
\newcommand{\Var}{\mathop{\mathbb{V}\mathrm{ar}}}
\newcommand{\conn}{K}
\newcommand{\compact}{\mathcal{K}}

\newcommand{\mF}{\mathcal{F}}

\newcommand{\fp}{\Phi}

\newcommand{\1}{\mathbh{1}}

\makeatother

\begin{document}
\begin{frontmatter}

\title{A counterexample to the Cantelli conjecture through the
Skorokhod embedding problem}
\runtitle{A counterexample to the Cantelli conjecture}

\begin{aug}
\author[A]{\fnms{Victor}~\snm{Kleptsyn}\thanksref{T1}\ead[label=e1]{victor.kleptsyn@univ-rennes1.fr}}
\and
\author[B]{\fnms{Aline} \snm{Kurtzmann}\corref{}\ead[label=e3]{aline.kurtzmann@univ-lorraine.fr}}

\thankstext{T1}{Supported in part of RFBR Grant 10-01-00739-a and
joint RFBR/CNRS Grant 10-01-93115-CNRS-a.}
\affiliation{Universit\'e Rennes 1 and Universit\'e de Lorraine}

\address[A]{IRMAR (UMR 6625 CNRS)\\
Universit\'e Rennes 1\\
Campus de Beaulieu\\
F-35042 Rennes Cedex\\
France\\
\printead{e1}}

\address[B]{Institut Elie Cartan de Lorraine\\
UMR 7502\\
Universit\'e de Lorraine\\
Vandoeuvre-l\`es-Nancy, F-54506\\
France\\
and\\
CNRS\\
Institut Elie Cartan de Lorraine\\
UMR 7502\\
Vandoeuvre-l\`es-Nancy, F-54506\\
France\\
\printead{e3}}
\runauthor{V. Kleptsyn and A. Kurtzmann}
\end{aug}

\received{\smonth{6} \syear{2013}}
\revised{\smonth{3} \syear{2014}}

%
\begin{abstract}
In this paper, we construct a counterexample to a question by Cantelli,
asking whether there exists a nonconstant positive measurable function
$\varphi$ such that for i.i.d. r.v. $X,Y$ of law $\mathcal{N}(0,1)$, the r.v.
$X+\varphi(X)\cdot Y$ is also Gaussian.

This construction is made by finding an unusual solution to the
Skorokhod embedding problem (showing that the corresponding Brownian
transport, contrary to the Root barrier, is not unique). To find it, we
establish some sufficient conditions for the continuity of the Root
barrier function.
\end{abstract}

%

%

\end{frontmatter}

\section{Introduction}\label{s:intro}
\subsection{History of the Cantelli conjecture}
The general thema of this paper is the following.

\begin{conjecture*}
Let $X,Y$ be two real random variables, of standard Gaussian
distribution law. Suppose that $X$ and $Y$ are independent. Let
$\varphi
$ be a measurable \emph{nonnegative} function. Then the random
variable $X+\varphi(X)\cdot Y$ has a Gaussian distribution law if and
only if $\varphi$ is constant.
\end{conjecture*}

Actually, Cantelli has originally mentioned this as a question in his
paper \cite{C}, page 407, asking whether it is possible to have a
nonconstant function $\varphi$, but later it became known as Cantelli
conjecture.
This conjecture has been previously studied by different authors.
First, Tortorici \cite{To} has given some restrictions on the function
$\varphi$ to satisfy the conjecture. To do that, he has developed
$\varphi$ in a Hermite series and has approached the solution (via a
truncation of the series). Then Tricomi \cite{T} has used analytical
tools in order to describe some properties satisfied by the function
$\varphi$ (through the characteristic function). In the same paper, he
has also given a survey on this subject.
Later, Dudley \cite{D} has exposed two unsolved problems about
finite-dimensional Gaussian measures. One of them was Cantelli
conjecture. Dudley said about it ``The problem seems to be a mere
curiosity, but that will perhaps be unclear until it is solved.'' Letac
has also worked on this problem and has emphasized this question in his
exercise book with Malliavin \cite{LM}. Indeed, they have suggested an
exercise, showing that the decomposition of $\varphi$ with respect to
the Hermite polynomials,\vspace*{-2pt} that is, $\varphi(x) = \sum_{n\ge0} \varphi_n
\frac{H_n(x)}{n!}$ [in the $L^2(e^{-x^2/2} \frac{\mathrm{d}x}{\sqrt
{2\pi
}})$ sense] is such that $\varphi_1=0$, $-2\varphi_2 = \sum_{n\ge2}
\frac{\varphi_n^2}{n!}$ and $\varphi(x) \le\varphi_0+1$ almost everywhere.

Finally, this striking question has been mentioned by de Meyer,
Roynette, Vallois and Yor \cite{DRVY}, Section 6. Actually, they
answered a related question, asked by Tortrat. Consider a standard
$(\mF
_t, t\ge0)$-Brownian motion, denoted by $(B_t, t\ge0)$. Can one
find an a.s. bounded random variable $Z$, nonconstant and $\mF
_1$-measurable, such that $B_1+Z(B_2-B_1)$ has a Gaussian distribution law?
de Meyer et al. have proved the existence of a linear standard $(\mF_t,
t\ge0)$-Brownian motion $(B_t, t\ge0)$, and a stopping time
$T$ [w.r.t. $(\mF_t, t\ge0)$] which is bounded by 1, nonconstant
and such that $B_T$ has a Gaussian distribution law. Thanks to this
result, they have shown that the random variable $B_1 + \sqrt
{T}(B_2-B_1)$ has a Gaussian distribution law. In their example, $\sqrt
{T}$ is $\mF_1$-measurable, bounded and nonconstant. However, $\sqrt
{T}$ is not a function of $B_1$. So this construction does not
contradict the Cantelli conjecture.

In the present paper, we construct a counterexample to the Cantelli
conjecture. It seems interesting to us (being, perhaps, a reply to a
phrase of Dudley cited above), that its construction uses the link of
the question to the Skorokhod embedding problem, as well as to the
Stefan-type problem in partial differential equations.

Let us indicate how the rest of this paper is organized. The first step
in the construction of a counterexample to the Cantelli conjecture,
stated in Section \ref{ss:main}, is based upon its link with the other
famous problem, the Skorokhod embedding problem for which we remind the
preceding works in Section \ref{ss:history}. We will explain the link
below in Section \ref{ss:BT}. Also, we will introduce there a notion
closely related to Skorokhod embedding-type problems (in particular to
Root barrier): the Brownian transport. For our construction to work, we
need some existence statements about this transport: Theorems \ref
{thm:BT-line} and \ref{thm:BT-interval}. These theorems are stated in
Section~\ref{s:skoro}.

The main tool in the proof of Theorems \ref{thm:BT-line} and \ref
{thm:BT-interval} is the potential function $\fp_t$ (going back to
Chacon \cite{cha} and obeying a PDE of the type of Stefan problem),
that we introduce in Section \ref{s:ideas}. Using this function, we
obtain some {a priori} estimates. Roughly speaking, ``how the
solution should look like assuming that it is nice.'' We also deduce
from these estimates Theorem \ref{t:main} (that will thus be
established once these estimates are formally proven).

Finally, in Section \ref{s:BT-exists}, by means of the discretization,
we prove the {a priori} estimates, thus completing the proofs of
our result.

\subsection{Result for the Cantelli conjecture}\label{ss:main}

Our main result here will be the following.
%

\begin{theorem}\label{t:main}
There exists a measurable nonconstant function $\varphi\dvtx\R\to\R_+$
such that for two independent standard Gaussian variables $X,Y\sim\mN
(0,1)$, the random variable
$X+\varphi(X)\cdot Y$ is also Gaussian.
\end{theorem}

In fact, as we will see from the construction in Section \ref
{s:construction}, the function $\varphi$ can be taken to be a
``choice'' between two continuous functions:
\[
\varphi(x)= %
\cases{ \varphi_0(x), & \quad$x\in\compact$,
\vspace*{2pt}
\cr
\varphi_1(x), &\quad $x\notin\compact$,} %
\]
where $\compact$ is a fat Cantor set of positive Lebesgue measure (see
its construction in Section \ref{ss:construction2}) and $\varphi
_0,\varphi_1\in C(\R)$. Actually, the function $\varphi$ we construct
here is discontinuous. We believe that Cantelli conjecture is true if
we impose the continuity of the function $\varphi$, but we have no
proof for that.

\subsection{The Skorokhod embedding problem: Historical context}\label
{ss:history}
The Skorok\-hod embedding problem is the following. For a given centered
probability measure $\mu$ with finite second moment and a Brownian
motion $B$, one looks for an (integrable) stopping time $T$ such that
the distribution law of $B_T$ is $\mu$. Several authors have developed
different techniques to solve this problem, which has stimulated
research in probability theory since the first formulation of
Skorokhod \cite{S}; we present briefly here their results that we need,
largely (except for those appeared after its publication date)
following an excellent survey by Ob\l{}oj \cite{O} (to which we refer
interested reader for more details).


One of the techniques closely related to our problem is Root's barrier,
introduced by Root in \cite{root}. Namely, he suggested to look for the
solution $T(\omega)$ in the form of the moment of the first
intersection of a Brownian trajectory $(B_t(\omega),t)$ with a
\emph{barrier}, that is a supergraph $\{(x,t)\dvtx t\ge f(x)\}$ of some lower
semicontinuous function $f\dvtx\R\to\R_+\cup\{+\infty\}$. He
proved the
(implicit) existence of such a barrier, establishing it with
topological arguments for a finitely supported target measure, and then
passing to the limit.
Soon afterward, Loynes \cite{L} has shown the uniqueness of the Root
barrier. Then Rost \cite{rost} has introduced the concept of one
measure being ``earlier'' than an other one, that is, a Brownian motion
starting with $\mu_0$ can be stopped with the law $\mu_1$, introducing
a filling process to check it.

Chacon \cite{cha} has introduced the notion of potential $U$. It turns
out that the convolutions $U_t$ of the function $|x|$ with the
occupation measures $\mu_t$ at time $t$ of a martingale $X_t$ (in
particular for a Brownian motion stopped at time $T$) form a monotonous
family of functions. McConnell \cite{MC} related these potential
functions to the Stefan problem: a particular type of a PDE, introduced
in 1831 by Lam\'e and Clapeyron as a model of melting ice (see the
survey of Vuik \cite{V} for details).

In his seemingly unpublished work, Rost has considered inverse barriers
(see~\cite{M} or \cite{O}, Section 7.3). Such barriers have also been
studied in \cite{CoxP}. Cox and Wang \cite{CoxW} have further studied
Root barriers, in particular, developing the case of a non-Dirac
initial measure $\mu_0$. They have also studied the Stefan-type PDE
relating the potential and the barrier. Finally, Ankirchner and
Starck \cite{AS} have studied the conditions for the stopping time to
be bounded.

\section{Construction} \label{s:construction}

\subsection{Construction: First step}\label{ss:construction1}
The first step in the proof of Theorem \ref{t:main} is the following
idea, close to \cite{DRVY}. Consider the standard Brownian motion
$(B_t, t\ge0)$, and let $T=T(\om)$ be a stopping time [w.r.t. the
standard family $(\mF_t, t\ge0)$ of $\sigma$-algebras], such that
$T<C$ almost surely for some constant $C$. Then
%
%
\begin{equation}
\label{eq:BT} B_C=B_T+ (B_C-B_T)
= B_T + \sqrt{C-T} \cdot\xi,
\end{equation}
where the random variable $\xi:=\frac{B_C-B_T}{\sqrt{C-T}}$ is a
standard Gaussian variable $\mN(0,1)$ and is independent from $B_T$ due
to the Markov property.

Now note that $B_C$ is a Gaussian random variable, so
%
%
\begin{equation}
B_T + \sqrt{C-T} \cdot\xi\sim\mN(0,C),\qquad B_T \perp\!\!\!
\perp\xi, \xi\sim\mN(0,1).
\end{equation}
Compare it to what we need to prove Theorem \ref{t:main} (and hence to
disprove the Cantelli conjecture):
%
%
\begin{equation}
X+\varphi(X)\cdot Y \sim\mN(0,\cdot),\qquad X\perp\!\!\!\perp Y, X,Y\sim
\mN(0,1).
\end{equation}
This comparison immediately gives us the following conclusion.
%

\begin{proposition}\label{p:time}
Let $T=T(\om)$ be a nonconstant stopping time for the standard Brownian
motion $(B_t, t\ge0)$, and assume that the following holds:
\begin{longlist}[(iii)]
\item[(i)]$\exists C\dvtx\forall\om T(\om)<C$;
\item[(ii)] The law of $B_T$ is the standard Gaussian law: $B_T\sim\mN(0,1)$;
\item[(iii)]
There exists a measurable function $f\dvtx\R
\to\R
_+$, such that almost surely $T=f(B_T)$.
\end{longlist}
Then the function $\varphi(x)=\sqrt{C-f(x)}$ provides us a
counterexample to the Cantelli conjecture.
\end{proposition}

Indeed, using the latter result that will prove Theorem \ref{t:main} in
Section \ref{s:ideas}, we will construct a nonconstant stopping time
satisfying the assumptions of Proposition~\ref{p:time}.
%

\begin{remark*}
There is one subtlety with property (iii) that we would
like to emphasize. While this property \emph{says} that the stopping
moment $T$ should be equal to a function of the place $B_T$ where the
process was stopped, it \emph{does not} say that we should stop the
process immediately once the equality $t=f(B_t)$ is satisfied.
Moreover, for the construction in the proof of Theorem \ref{t:main}, it
is \emph{not} true that $T=\min\{t\colon t=f(B_t) \}$.
\end{remark*}

\subsection{Brownian transport} \label{ss:BT}

Proposition \ref{p:time} naturally leads us to the following definition.
%

\begin{definition}
Let $\mu_0,\mu_1$ be two probability measures, with the same mean and
square integrable. We say that there exists a \emph{Brownian transport} from $\mu_0$ to $\mu_1$ if, for a random process $(X_t, t\ge0)$ such
that $X_0\sim\mu_0$ and $\mathrm{d}X_t =\mathrm{d}B_t$ (where $B$
is a
real Brownian motion independent of $\mu_0$), one can find a stopping
time $T$ and a function $f$ such that:
\begin{longlist}[(ii)]
\item[(i)]$X_T \sim\mu_1$,
\item[(ii)] a.s. $T=f(X_T)$.
\end{longlist}
We say that $f$ is the \emph{stopping function} of this transport.

If the stopping time $T$ has finite expectation, then we say that there
exists a \emph{finite expectation Brownian transport}.

If the function $f$ is bounded or continuous, we speak, respectively, of
\emph{bounded} or \emph{continuous Brownian transport} from $\mu_0$
to $\mu_1$.
\end{definition}

%
\begin{remark*}
Moreover, if the function $f$, corresponding to a bounded Brownian
transport, can be taken to be continuous, then the moment $T$ is the
first intersection time of the trajectory $(X_t, t\ge0)$ with the
graph of $f$:
\[
T(\om) = \inf\bigl\{t\ge0\dvtx t=f(X_t)\bigr\}.
\]
\end{remark*}

In other words, the case of a continuous bounded Brownian transport is
always described by a Root barrier (see \cite{root}) for the
corresponding Skorokhod problem. In this case, such a transport is
unique (due to Loynes), though both assumptions (continuity and
boundedness) are essential. An unbounded solution can easily correspond
to, for instance, Rost inverse barrier (see \cite{rost1}). Moreover, it
can be shown that there exist square integrable measures $\mu$ given by
the Rost solution corresponding to the inverse barrier $\{t\le\varphi
(x)\}$ with a continuous sufficiently quickly growing function $\varphi
$. On the other hand, such $\mu$ can be chosen to fulfill the
assumptions of Theorem \ref{thm:BT-line} below, and thus can also be
obtained by a continuous Brownian transport corresponding to the Root
barrier solution. This shows nonuniqueness of a continuous Brownian
transport, even with the additional assumption of finiteness of
expectation. Finally, the construction we propose in Section \ref
{ss:construction2} shows that bounded Brownian transport (without the
assumption of continuity of the stopping function) is highly nonunique.

Note also that a bounded Brownian transport between two given square
integrable measures $\mu_0,\mu_1$ does not always exist. An obvious
restriction for its existence is that one should necessarily have $\E
\mu_0 = \E\mu_1$ and $\Var\mu_0 \le\Var\mu_1$, though this
condition is far from being sufficient. For instance, one can easily
see that $\mu_1$ cannot have atoms (unless $\mu_0$ charges the same
points with at least the same mass), and that the bounded Brownian
transport cannot create ``holes'' inside the support: a necessary
condition is $\operatorname{Supp}(\mu_0) \subset\operatorname{Supp}(\mu_1)$.

A finer, but much more restrictive, necessary condition is that the
potential functions $\fp_{\mu_0}$ and $\fp_{\mu_1}$ (in the sense of
Section \ref{s:ideas}, or what is almost the same, of Chacon~\cite{cha}
and Cox--Wang \cite{CoxW}), corresponding, respectively, to $\mu_0$ and
$\mu_1$, should satisfy $\fp_{\mu_0} \le\fp_{\mu_1}$ on the real line.

Finally, even such a positivity and the condition on the supports are
not sufficient: taking the measure $\mu_1$ to be the first intersection
measure with the graph $\{t=\frac{1}{|x|}\}$ of the function $\varphi
(x) =\frac{1}{|x|}$,\vspace*{1pt} we see that (due to the uniqueness by Loynes)
there is no continuous bounded Brownian transport for such a $\mu_1$.
Moreover, from~\cite{AS}, one sees that a necessary condition for a
bounded Brownian transport from $\delta_0$ to $\mu_1$ to exist is that
there are no ``too weakly charged'' intervals for $\mu_1$ [compare with
assumptions (iii) of Theorems \ref{thm:BT-line} and \ref{thm:BT-interval}].

However, in Section \ref{s:ideas}, we will state two theorems
establishing sufficient conditions for the existence of a continuous
finite expectation Brownian transport on an interval and on the real line.

\subsection{Construction: Second step}\label{ss:construction2}
We can now describe how the stopping time~$T$, satisfying the
assumptions of Proposition \ref{p:time}, will be constructed. We will
fix a moment $t_0\in(0,1)$ and choose in a small neighborhood of the
origin a fat Cantor set $\compact\subset\R$ of positive Lebesgue
measure (with some restrictions on its geometry), such that on this set
the density of the law $\mN(0,1)$ is everywhere upper bounded by the
density of the law $\mN(0,t_0)$:
\[
\rho_{\mN(0,t_0)}(x)> \rho_{\mN(0,1)}(x)\qquad \forall x\in\compact.
\]
Then, at the moment $t_0$, for any $x\in\compact$, we stop the
proportion $\frac{\rho_{\mN(0,1)}(x)}{\rho_{\mN(0,t_0)}(x)}$ of
all the
trajectories passing through $x$ at this moment. To do so, one can
either use a probabilistic Markov time, modifying the initial
probability space of the Brownian motion by multiplying it by $[0,1]$, or
note that the random variable $S_{t_0}(\om):=\sup_{0\le t\le t_0}
|B_t(\om)|$ has a continuous conditional distribution w.r.t. any
condition $B_{t_0}=x$, and hence, denoting by $\vk(\alpha,x)$ the
$\alpha$-quantile of the corresponding conditional distribution [that
is the value $y$ such that $\bbP(S_{t_0}\le y| B_{t_0}=x) \ge\frac
{1}{\alpha}$], we can put
%
%
\begin{equation}
\qquad T(\om) = t_0\qquad \mbox{if } x:=B_{t_0}(\om)\in\compact\mbox{ and } S_{t_0}(\om) \le\vk\biggl(\frac{\rho_{\mN(0,1)}(x)}{\rho
_{\mN
(0,t_0)}(x)},x \biggr).
\end{equation}
This stopping ensures that the transport time $T$ and the corresponding
function $f$ are nonconstant: there is something left to transport.

The following problem now remains. At the moment $t_0$, there is a
conditional distribution of not yet stopped trajectories, with the density
%
%
\begin{equation}
\label{eq:rho} \rho_0(x)= %
\cases{ c^{-1}
\rho_{\mN(0,t_0)}(x), &\quad $x\notin\compact$, \vspace*{2pt}
\cr
c^{-1}
\bigl(\rho_{\mN(0,t_0)}(x)-\rho_{\mN(0,1)}(x) \bigr), &\quad $x\in
\compact$,}
\end{equation}
where $c=\mathbb{P}(\mathcal{N}(0,1) \notin\compact)$.
We want to stop these trajectories at a bounded stopping time $T$, such that:
\begin{longlist}[(ii)]
\item[(i)]$T=f(B_T)$,
\item[(ii)] the law of $B_T$ conditionally to $T>t_0$ is the restriction (to
$\R\setminus\compact$) of the standard Gaussian law $\mN(0,1)|_{\R
\setminus\compact}$.
\end{longlist}

In other words, we are looking for a solution of the following.
%

\begin{pb}\label{pb:1}
Find a bounded Brownian transport from $\mu_0=\rho_0 \,\mathrm{d}x$,
given by~\eqref{eq:rho}, to $\mu_1$ which is the conditional
distribution of $\mN(0,1)$ on $\R\setminus\compact$.
\end{pb}

Indeed, once Problem \ref{pb:1} is solved with the bounded stopping
time $T_1$ such that $T_1 = f_1(B_{T_1})$, we can take for the original problem
%
%
\begin{equation}
\label{eq:T} T(\om)= %
\cases{ t_0, & \quad $\mbox{if }
x:=B_{t_0}(\om)\in\compact\mbox{ and}$\vspace*{2pt}\cr
&\quad$\displaystyle S_{t_0}(\om) \le\vk
\biggl(\frac{\rho_{\mN(0,1)}(x)}{\rho_{\mN(0,t_0)}(x)},x\biggr
)$, \vspace*{2pt}
\cr
t_0+T_1,
&\quad $\mbox{otherwise}$,} %
\end{equation}
where $T_1$ is evaluated on the trajectory $X_t = B_{t_0+t}$. We then have
%
%
\begin{equation}
\label{eq:f} f(x)= %
\cases{ t_0, &\quad $\mbox{if } x\in
\compact$, \vspace*{2pt}
\cr
t_0+f_1(x), &\quad $\mbox{if } x
\notin\compact$. } %
\end{equation}

%
\begin{remark*}
It is important to note that, due to the choice of the ``target
measure'' $\mu_1$, the stopping point of the process $(X_t,t\ge0)$ a.s.
does not belong to $\compact$. Hence, even though in \eqref{eq:f}, the
function $f$ on $\compact$ does not coincide with $t_0+f_1(x)$, the
equality $T=f(B_T)$ still a.s. holds for the trajectories not yet
stopped at time~$t_0$.
\end{remark*}

To solve Problem \ref{pb:1}, we prove a sufficient condition for a more
general result (that we have already mentioned in Section \ref{ss:BT}),
establishing the continuity (and thus local boundedness) of the
corresponding Root barrier. Then, further studying the barrier function
$f_1$ in this particular case, we show that this function has a limit
at infinity, and thus is globally bounded. This proves the following.
%

\begin{theorem}\label{t:sol-pb}
Assume that $\compact\subset[-1,1]$ and that there exists $\alpha>0$
such that, for any interval $I\subset[-1,1]$, one has $\operatorname{Leb}(I\setminus
\compact) \ge\exp\{-\alpha/ |I|\}$. Then there exists a solution
$T_1$ to Problem \ref{pb:1} and the corresponding function $f_1$ is
continuous. Moreover, $T_1$ can be represented as a ``first
intersection'' moment
\[
T_1(\om) = \inf\bigl\{t\ge0\dvtx t=f_1(X_t)
\bigr\}.
\]
\end{theorem}

Figure \ref{fig:5} shows a simulation of the functions $f_1$ and
$\varphi$ (that one can do thanks to an almost explicit nature of our
construction).

%
\begin{figure}[b]

\includegraphics{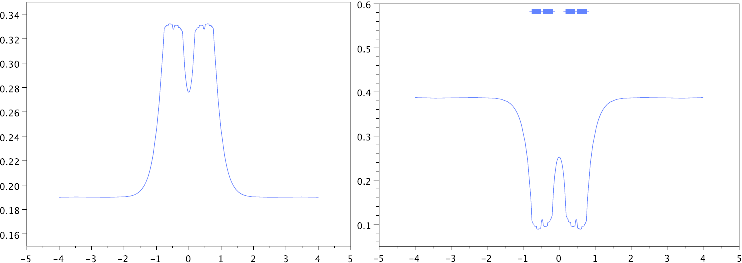}

\caption{On the left: the graph of the function $f_1$. On the right:
the graph of the resulting function~$\varphi$.}\label{fig:5}
\end{figure}

It is not difficult to construct a compact set $\compact$ satisfying
the assumptions of Theorem \ref{t:sol-pb}. Actually, if in the standard
construction of the Cantor set, one chooses to remove on the $n$th step
an $\frac{1}{(n+1)^2}$th part around the middle of the previously
constructed intervals, the obtained Cantor set $\compact$ satisfies
these assumptions. Moreover, for this Cantor set, an even stronger
estimate holds: $\operatorname{Leb}(I\setminus\compact) \ge\alpha|I|^2$ for some
universal constant $\alpha$.

Once such a set $\compact$ is constructed, the above arguments allow us
to deduce Theorem \ref{t:main} from Theorem \ref{t:sol-pb}. So, the
task of disproving the Cantelli conjecture is reduced to proving
Theorem \ref{t:sol-pb}.

\subsection{Results for the transport problem}\label{s:skoro}

Even though our stopping times appearing in Theorem \ref{t:sol-pb} as
well as in Theorems \ref{thm:BT-line} and \ref{thm:BT-interval} below
\emph{are} (due to the uniqueness by Loynes) Root stopping times, we
cannot obtain their existence directly from Root's result. The problem
here is that we need a \emph{bounded} (and preferably \emph{continuous}) stopping function, and Root's function is only lower semicontinuous.

The second main result of the paper is the following.

%
\begin{theorem}\label{thm:BT-line}
Let $\mu_0,\mu_1$ be two centered probability measures, square
integrable and which support is $\R$. Suppose that, for any $R$ large
enough, the troncated probability measures $\tilde{\mu}_0^R = \frac
{\mu
_0|_{[-R,R]}}{\mu_0([-R,R])}$ and $\tilde{\mu}_1^R = \frac{\mu
_1|_{[-R,R]}}{\mu_1([-R,R])}$ satisfy:
\begin{longlist}[(iii)]
\item[(i)]$\tilde{\mu}_0^R$ and $\tilde{\mu}_1^R$ are absolutely continuous
with respective densities $\rho_{\mu_0}$ and~$\rho_{\mu_1}$,
\item[(ii)] there exist $a_R,b_R>0$ such that for all $-R\le x\le R$, we have
$\rho_{\mu_0}(x) \ge a_R$ and $\rho_{\mu_1}(x) \le b_R$,
\item[(iii)] there exists $\alpha_R>0$ such that for any $J\subset[-R,R]$, we
have $\mu_1(J) \ge e^{-\alpha_R /|J|}$.
\end{longlist}
Assume also that:
\begin{longlist}[(iii)]
\item[(iv)] for any $x\in\R$, we have $\fp_{\mu_0\rightarrow\mu_1}(x):= \int_{-\infty}^x (\mu_0-\mu_1)((-\infty,s]) \,\mathrm{d}s>0$,
\item[(v)] $\limsup_{|x|\rightarrow+\infty} \frac{\rho_{\mu
_0}(x)}{\rho
_{\mu_1}(x)} <1$.
\end{longlist}
Then there exists a finite expectation continuous Brownian transport
from $\mu_0$ to $\mu_1$, with a possibly unbounded stopping time $T$.
Moreover, this Brownian transport is given by the first intersection
time with the graph the stopping function $f$.
\end{theorem}


%
\begin{remark*}
We can actually suppose in the latter theorem that the measures $\mu
_0,\mu_1$ have the same mean (instead of being centered).
\end{remark*}

An analogous question can be also asked for measures supported on an
interval. This question, on one hand, turns out to be a bit simpler
than the real line one (due to the compactness and lack of effects at
infinity). On the other hand, it becomes one of the steps in our proof
of Theorem \ref{thm:BT-line}: the function $f$ is constructed as a
limit of a subsequence of functions $f_R$ corresponding to a
``cut-off'' problem. The corresponding theorem is
the following.
%

\begin{theorem}\label{thm:BT-interval}
Let $\mu_0,\mu_1$ be two probability measures, with the same mean,
square integrable and which support is an interval $I\subset\R$.
Suppose that they satisfy the hypotheses:
\begin{longlist}[(iii)]
\item[(i)]$\mu_0$, $\mu_1$ are absolutely continuous with respective
densities $\rho_{\mu_0}$, $\rho_{\mu_1}$,
\item[(ii)] there exist $a,b>0$ such that for all $x\in I$, we have $\rho
_{\mu
_0}(x) \ge a$ and $\rho_{\mu_1}(x) \le b$,
\item[(iii)]
there exists $\alpha>0$ such that for any interval
$J\subset I$, we have $\mu_1(J) \ge e^{-\alpha/|J|}$,
\item[(iv)] for all $x\in I$, we have $\fp_{\mu_0\rightarrow\mu_1}(x):=
\int_{-\infty}^x (\mu_0-\mu_1)((-\infty,s]) \,\mathrm{d}s>0$,
\item[(v)]$\rho_{\mu_0}>\rho_{\mu_1}$ in some inner neighborhood
$\mathcal
{U}_\varepsilon(\partial I)\cap I$.
\end{longlist}
Then there exists a bounded Brownian transport from $\mu_0$ to $\mu_1$,
given by the first intersection time with the graph of some continuous
function $f$.
\end{theorem}

In other words, under the respective assumptions, Theorems \ref
{thm:BT-line} and \ref{thm:BT-interval} state that the Root barrier
corresponding to the transport of $\mu_0$ to $\mu_1$ is given by a
continuous function.

The proof of these two results will be done in several steps. First, we
will do some a priori estimates and transformations, answering the
question ``assuming that such a transport exists, how should it look
like?'' The understanding coming from these steps will leave us with
some kind of a PDE problem, of the Stefan type.

However, we could not establish the existence theorems for this problem
directly by PDE methods (in fact, it seems to be an interesting
question to us), we establish them via a discretization procedure: we
solve an analogous discrete problem and pass to the limit as the mesh
goes to 0. This part is rather technical and is postponed
to Section \ref{s:B-transport-interval}.

%
\begin{remark*}
Some assumptions of Theorems \ref{thm:BT-line} and \ref
{thm:BT-interval} seem nonrestrictive, such as the positivity of $\fp
_{\mu_0 \rightarrow\mu_1}$ (inside $I$ for Theorem \ref
{thm:BT-interval}). Indeed, a necessary condition is that the function
$\fp$ is nonnegative (see Corollary \ref{l:pos-phi}). Though, in the
case of a nonnegative function $\fp$ that is not positive everywhere
inside $I$, one can simply split the interval $I$ into the intervals of
positivity of $\fp$ (see Lemma \ref{l:decomp-I}).
Other assumptions, such as (iii), seem unavoidable in order to
assure the uniform boundedness of the stopping time. Indeed, otherwise
the first intersection measure of the Brownian motion with the graph\vspace*{1pt} of
$f(x) = \frac{1}{|x|}$ would satisfy the assumptions of the theorem.
Finally, some assumptions (such as the absolute continuity of $\mu_0$
or the lower bound for its density) could be weakened. 
But we are not doing it in the present work: the statement of
Theorem \ref{thm:BT-interval} suffices for our construction.
\end{remark*}

\section{Tools: The potential function \texorpdfstring{$\fp$}{Phi} and some a priori
arguments}\label{s:ideas}

In the following, for a regular time--space function $\fp\dvtx \R
_+\times\R
\rightarrow\R$, $(t,x) \mapsto\fp_t(x)$, we denote by $\dot{\fp}_t(x)
= \partial_t \fp_t(x)$ its time-derivative.
Moreover, for an absolutely continuous measure $\mu$, we denote by
$\rho
_\mu$ its density distribution function. The $\varepsilon$-neighborhood
of a set $I$ is denoted by $\mathcal{U}_\varepsilon(I)$. As all the
objects we consider in this section are invariant by a translation, we
will suppose that the measures $\mu_0, \mu_1$ are centered.

\subsection{The potential function \texorpdfstring{$\fp$}{Phi} and Stefan-type
problem}\label
{ss:stefan}
Before going deeper into the proof of the existence theorems
(Theorems \ref{t:sol-pb}, \ref{thm:BT-line} and \ref{thm:BT-interval}),
let us first do some a priori arguments. Namely, assuming that a finite
expectation Brownian transport from some centered measure $\mu_0$ to
some other centered measure $\mu_1$ exists (both $\mu_0$, $\mu_1$
having a finite second moment), let us find out what could be its
properties and how could it be described.

Chacon has introduced in \cite{cha} the potential $U$, that is the
convolutions of the function $|x|$ with the occupation measures at
time $t$ of a martingale. In our setting, the following definition,
corresponding to the convolutions of the function $|x|_+:=\max(0,x)$
with the occupation measures at time $t$ seems to be easier to work
with (though, they are related with an affine change).
%

\begin{definition}
Let $\mu$ be a measure on $\R$, with finite second moment. Then we
denote by $\fp_\mu$ the primitive of its distribution function $F_\mu
(x):= \mu((-\infty,x])$:
%
%
\begin{equation}
\label{eq:phi_mu} \fp_\mu(x):= \int_{-\infty}^x
\mu\bigl((-\infty,s]\bigr) \,\mathrm{d}s = |x|_+ * \mu.
\end{equation}
\end{definition}

An easy computation then shows that
%
%
\begin{equation}
\label{eq:def-phi_mu} \fp_\mu(x) 
= x -
\E(\mu) +\int_x^{+\infty} \mu\bigl([s,+\infty)\bigr)
\,\mathrm{d}s.
\end{equation}

In particular, for any two such measures $\mu_0,\mu_1$ with the same
mean, the difference between the corresponding functions
%
%
\begin{equation}
\label{eq:def-diff-phi} \fp_{\mu_0 \rightarrow\mu_1} (x):= \fp
_{\mu_1}(x) -
\fp_{\mu_0}(x)
\end{equation}
converges to 0 as $x$ tends to $-\infty$ and as $x\to+\infty$. [A
reader familiar with the Chacon's potential easily notices that $\fp
_{\mu_0 \rightarrow\mu_1} =\frac{1}{2}(U_{\mu_1}-U_{\mu_0})$ due to
the affine relation between $\Phi_{\mu}$ and $U_{\mu}$.]

The role of $\fp$ is then given by the following conclusions, going
back to Chacon~\cite{cha}. Let $(X_t,T)$ be a finite expectation
Brownian transport from $\mu_0$ to $\mu_1$. Denote by $\tilde{X}_t:=
X_{t\wedge T}$ the ``stopped'' process, by $\tilde{\nu}_t$ its
distribution law at time $t$, and by $\nu_t$ the (nonprobability)
measure given by the ``not yet stopped'' particles: for any Borel set
$A$, we have
\[
\nu_t (A) = \bbP(X_t \in A, t<T).
\]
%

\begin{lemma}\label{eq:derive-phi}
$
\dot{\fp}_{\tilde{\nu}_t}= \frac{1}{2}\rho_{\nu_t}$.
\end{lemma}

\begin{pf}
Indeed, we have $\mathrm{d}\tilde{X}_t = \1_{t<T} \,\mathrm{d}B_t$ and
hence by the heat equation, we have
$
\dot{\fp}_{\tilde{\nu}_t}= \frac{1}{2}\rho_{\nu_t}$.
\end{pf}

An immediate corollary to this lemma is the following.
%

\begin{corollary}\label{l:pos-phi}
Let $\mu_0,\mu_1$ be two centered absolutely continuous probability
measures, with finite second moment. Suppose that there exists a finite
expectation Brownian transport from $\mu_0$ to $\mu_1$. Then, for any
$x\in\R$, we have $\fp_{\mu_0 \rightarrow\mu_1}(x) \ge0$.
\end{corollary}

%
\begin{figure}[b]

\includegraphics{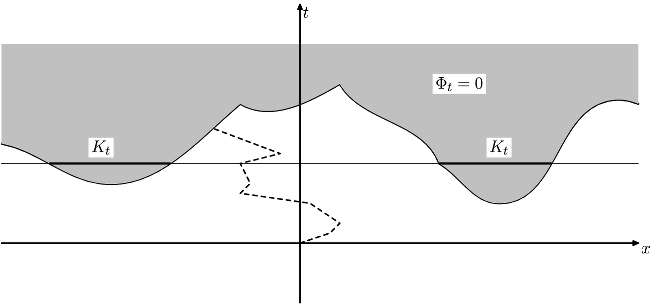}

\caption{Construction of $\conn_t$.}\label{fig:1}
\end{figure}

\begin{pf}
It is obvious from Lemma \ref{eq:derive-phi} that the functions $\fp_t
(x):=\break   \fp_{\tilde{\nu}_t \rightarrow\mu_1} (x) =\fp_{\mu_1} (x)
- \fp
_{\tilde{\nu}_t}(x)$ are monotonically decreasing with $t$ for any
fixed~$x$. The only thing we have to check is that $\fp_t (x)$
converges pointwise to $0$ (what is evident in the case of a bounded
Brownian transport, but needs to be justified in general). Indeed,
$\tilde{X}_t$ is a martingale and its variation
\[
\Var(\tilde{X}_t) = \Var(\tilde{X}_0) + \E(t\wedge T)
\le\Var(\mu_0) + \E T<\infty
\]
is uniformly bounded. Hence (see, e.g., \cite{BEK}, Theorem 4.3.3), we have that $\tilde{X}_t$ converges in $L^2$ to
$\tilde{X}_\infty(\om):= \lim_{t\to\infty} X_t (\om)$, and thus
\begin{eqnarray*}
\fp_{\tilde{\nu}_t} (x) &=& \int_{-\infty}^x \bbP(
\tilde{X}_t \le s) \,\mathrm{d}s = \int_\Omega\bigl|
\tilde{X}_t (\om) -s\bigr|_- \,\mathrm{d}\bbP(\om)
\\
&  \mathop{\longrightarrow}\limits_{t\to\infty}& \int_\Omega\bigl|\tilde{X}_\infty(
\om) -s\bigr|_- \,\mathrm{d}\bbP(\om) = \fp_{\mu_1}(x),
\end{eqnarray*}
where we have denoted $|x|_-:= |x|\cdot\1_{x\le0}$.
\end{pf}

These statements, in fact, suggest us a way of constructing the
stopping time $T$. Namely, together with the process $(X_t,t\ge0)$, we
consider an increasing family of closed sets $\conn_t = \{\fp_t =0\}$
(that will be in fact sections of the supergraph of $f$: $\conn_t = \{
x\in\R\dvtx t\ge f(x)\}$, as shown in Figure \ref{fig:1} below).\vadjust{\goodbreak} We stop
the process once it reaches this family:
\[
T = \inf\{t\ge0\dvtx X_t \in\conn_t\}.
\]
The function $f$ will then be defined as
\[
f(x) = \inf\{t\ge0\dvtx x\in\conn_t\} = \inf\bigl\{t\ge0\dvtx
\fp
_t(x) =0\bigr\}.
\]
Roughly speaking, we let the function $\fp_t = \fp_{\tilde{\nu}_t
\rightarrow\mu_1}$ decrease (as $\dot{\fp}_t \le0$), and once it
vanishes somewhere, we add this place to the set $\conn_t$ of ``stopped
motion.'' Due to this description, we will call in the future $\fp
_{\mu
_0 \to\mu_1}$ the \emph{potential function} of the finite expectation
Brownian transport from $\mu_0$ to $\mu_1$.

We wish to emphasize that the above description is absolutely
unrigorous. It cannot be used without proving the corresponding
existence theorems that do not seem to have an obvious direct proof.
So, we will prove them in Section \ref{s:BT-exists}, via the
discretization procedure. However, it gives an explanation why
Theorems~\ref{thm:BT-line} and~\ref{thm:BT-interval} should hold.

Moreover, this description can be (for the case of an absolutely
continuous measure $\mu_0$) rephrased in terms of Stefan-type problem.
Namely, the density $\rho_t = \rho_{\nu_t}$ obeys the heat equation
$\dot{\rho}_t = \frac{1}{2} \Delta\rho_t$ with the (moving) Dirichlet
boundary condition $\rho_t|_{\conn_t} =0$. So, the couple $(\fp_t,
\rho
_t)$ and the function $f(x)$ obey the system
%
%
\begin{equation}
\label{eq:stefan} %
\cases{ \dot{\fp}_t = -\frac{1}{2}
\rho_t,\vspace*{2pt}
\cr
\dot{\rho}_t =
\frac{1}{2} \Delta\rho_t,&\quad $\mbox{if } t<f(x),$\vspace*{2pt}
\cr
\fp_{f(x)}(x) =0,\vspace*{2pt}
\cr
\rho_{f(x)} (x)=0, } %
\end{equation}
where the third equation defines the function $f$, while the last one
is considered as a boundary condition on $\rho$.

We will not go deeper into giving fully formal sense to the
system \eqref{eq:stefan} (e.g., note that on the graph of $f$,
the derivative $\dot{\fp}_t$ can be discontinuous and if $f$ is
constant on some interval, then at the corresponding points, the
density $\rho$ will abruptly go to 0). As we have already mentioned in
Section \ref{s:skoro}, we could not prove the existence theorem here by
PDE methods, though it would be interesting to find such a direct
proof. However, we would like to emphasize here that the system \eqref
{eq:stefan} seems analogous to the Stefan problem of melting ice
(see \cite{rubi,V}).

Even though we have not yet established the existence of the process
described by the above rules, for the rest of this paragraph, we will---in
order to understand the ideas before passing to the technical
part---assume that it exists, and then study its behavior. Note that
one of the questions appearing (and that will be answered below) is the
following one: does $\fp_t$ vanish everywhere in finite time? To answer
this question, it is natural to consider the connected components of
$\R
\setminus\conn_t$ and to study their evolution. In fact, to prove
Theorem \ref{t:sol-pb}, we have to show that any of them disappears in
a finite time. This will be done in Lemma \ref{l:decomp-I}. The next
result deals with ``disconnecting'' different intervals from each
other, allowing us to study their evolution separately.

%
\begin{lemma}\label{l:coincidence}
Let $\mu,\tilde{\nu}$ be two centered absolutely continuous probability
measures on $\R$, with finite second moment, such that $\fp_{\tilde
{\nu
}\rightarrow\mu}$ is nonnegative. Let $x\in\R$ be such that $\fp
_{\tilde{\nu}\rightarrow\mu} (x)=0$. Then the measures $\mu$ and
$\tilde{\nu}$ of the interval $(-\infty,x]$ coincide, as well as the
expectations of the conditional measures $\frac{\tilde{\nu
}|_{(-\infty,x]}}{\tilde{\nu}((-\infty,x])}$ and $\frac{\mu|_{(-\infty,x]}}{\mu
((-\infty,x])}$.

The same holds for the restrictions on the interval $[x,+\infty)$ and
on any interval $[x,y]$ provided that $\fp_{\tilde{\nu} \rightarrow
\mu
}$ vanishes at both of its endpoints.
\end{lemma}

\begin{pf}
As the measures $\tilde{\nu}$ and $\mu$ are nonatomic, the function
$\fp
_{\tilde{\nu} \rightarrow\mu}$ is of class $C^1$. But, as $\fp
_{\tilde
{\nu} \rightarrow\mu}$ is nonnegative and $\fp_{\tilde{\nu}
\rightarrow\mu}(x)=0$, the point $x$ is a minimum of the function
$\fp
_{\tilde{\nu} \rightarrow\mu}$. Hence, $\partial_x \fp_{\tilde
{\nu}
\rightarrow\mu} (x) =0$. Noting that $\partial_x \fp_{\tilde{\nu}
\rightarrow\mu} (x) = -\mu((-\infty,x]) + \tilde{\nu} ((-\infty,x])$,
we obtain the first conclusion of the lemma. Now, remember
identity \eqref{eq:def-phi_mu}:
\[
\fp_\mu(x) = \int_{-\infty}^x (x-y)
\,\mathrm{d}\mu(y) = x \mu\bigl((-\infty,x]\bigr) - \int_{-\infty}^x
y \,\mathrm{d}\mu(y).
\]
As $\fp_{\tilde{\nu} \rightarrow\mu} (x)=0$, and thus $\fp_\mu
(x) =\fp
_{\tilde{\nu}}(x)$, we have
%
%
\begin{equation}
\label{eq:equality} x \mu\bigl((-\infty,x]\bigr) - \int_{-\infty}^x y
\,\mathrm{d}\mu(y) = x \tilde{\nu}\bigl((-\infty,x]\bigr) - \int_{-\infty}^x
y \,\mathrm{d}\tilde{\nu}(y).
\end{equation}
The equality between the first terms in the left- and right-hand sides
of \eqref{eq:equality} is already established, and thus implies the
equality between the last terms.

The other issues of the lemma are direct corollaries of the proved ones.
\end{pf}

\label{ss:ideas-line}
We are now ready to deduce Theorem \ref{t:sol-pb} from Theorem \ref
{thm:BT-line}. In other words, still assuming that the description
in Section \ref{ss:stefan} defines us the desired process, we conclude
the construction of the counterexample to the Cantelli conjecture. This
deduction will be split in several lemmas.

A first tool that we need is the following general lemma that allows to
estimate from above the time in which a connected component of $\R
\setminus\conn_t$ ``disappears.''
%

\begin{lemma}\label{l:decomp-I}
Let $(\tilde{X}_t,\conn_t)$ be constructed as described above
(Section \ref{ss:stefan}) for some probability measures $\mu_0,\mu
_1$ with the
same mean and finite second moment (but perhaps with no time $\bar{t}$
such that $\conn_{\bar{t}}=\R$). Let $I$ be an interval which is a
connected component of $\R\setminus\conn_t$ (at some time $t$).
Assume that for any interval $J\subset I$, we have $\mu_1(J) \ge\exp
\{
-\alpha/|J|\}$. Then there exists a constant $\theta$ (which does not
depend on $I$) such that $I\subset\conn_{t+\theta\alpha|I|}$.
\end{lemma}

\begin{pf}
We will first prove the following auxiliary statement: there exists a
constant $\theta_0$ such that, at the moment $t':=t+\theta_0\alpha
|I|$, any connected component of $I\setminus\conn_{t'}$ is of length
less than $|I|/2$. This statement will imply the conclusion of the
lemma. Indeed, applying it again to the connected components of
$I\setminus\conn_{t+\theta_0 \alpha|I|}$, we see that, at the moment
$t''= t+\theta_0 \alpha|I|+\frac{\theta_0}{2}\alpha|I|$, the lengths
of connected components of $I\setminus\conn_{t''}$ do not
exceed $\frac
{|I|}{4}$. We repeat this procedure. Thus, at the moment $t+2\theta_0
\alpha|I|$, we have $I \subset\conn_{t+2\theta_0\alpha|I|}$. This
completes the proof.

Let us now prove the latter statement. Indeed, note that for any
interval of complement $J\subset\R\setminus\conn_t$, the Wiener
measure of the trajectories that are still moving inside $J$ at the
time $t$ is equal to $\mu_1(J)$. Indeed, as $J$ is a connected
component of $\R\setminus\conn_t$, we have $\fp_t |_{\partial J} = 0$,
and hence Lemma \ref{l:coincidence} can be applied. So, to prove that
at some moment $t'>t$, the length of any connected component $J\subset
I\setminus\conn_{t'}$ is less than $|I|/2$, it suffices to show that,
at this moment, the proportion of trajectories that have not yet
intersected the graph of $f$ is at most $\exp(-\frac{\alpha}{|I|/2})$.

To do this, we consider a weaker stopping condition: the trajectory is
stopped once it reaches the boundary of $I$. The density of such a
process is given by the heat equation with the Dirichlet boundary
conditions on $I$. The measure of not yet stopped trajectories at the
moment $t+\tau$ is then given by the scalar product
$
\langle\varphi_\tau, 1/|I| \rangle$,
where
\[
\dot{\varphi}_\tau= \tfrac{1}{2} \Delta\varphi_{\tau},\qquad
\varphi_{\tau}|_{\partial I}=0,\qquad \varphi_0=
\rho_t.
\]
As the Laplace operator is self-adjoint, this scalar product is equal
to $\langle\psi_{\tau}, \varphi_0 \rangle$,
where
\[
\dot{\psi}_\tau= \frac{1}{2}\Delta\psi_{\tau},\qquad
\psi_{\tau
}|_{\partial I}=0,\qquad \psi_0=\frac{1}{|I|}.
\]

Thus, this scalar product does not exceed $|I|\cdot\sup_{I} \psi
_\tau
$. Rescaling the interval $I$ to $[0,1]$ and accordingly multiplying
the time by $1/|I|^2$ and the initial function by $|I|$, we obtain an
upper bound by
%
%
\begin{equation}
\label{eq:upp-bound} \sup_{[0,1]} \sum_{ n}
c_{2n+1} \exp\biggl\{-\frac{ \pi^2
(2n+1)^2}{2|I|^2} \tau\biggr\} \sin\bigl(
\pi(2n+1) x\bigr),
\end{equation}
where $c_{2n+1}=\frac{2}{2n+1}$ are the nonzero Fourier coefficients of
the function $1$ with respect to the eigenfunctions $\sin(\pi(2n+1)
x)$ of the Laplace operator on $[0,1]$. Estimating $c_n$ by 1 in \eqref
{eq:upp-bound} and the exponents by a geometric series, we see that
this supremum does not exceed
\[
\exp\biggl\{- \frac{\pi^2}{2|I|^2} \tau\biggr\} \cdot\frac
{1}{1-\exp\{-{\pi
^2}/{(|I|^2)} \tau\} }.
\]
Now, note that for $\tau= \frac{8}{\pi^2} \alpha|I|$, the first
factor is $\exp\{ - 4 \frac{\alpha}{|I|}\} = (\exp\{ - \frac
{\alpha}{|I|/2}\} )^2$. Thus, the product is at most
%
%
\begin{equation}
\label{eq:second} \exp\biggl\{- \frac{\alpha}{|I|/2}\biggr\}
\cdot\frac{\exp\{- {\alpha
}/{(|I|/2)}\}}{1-\exp\{- {\alpha}/{(|I|/2)}\} }.
\end{equation}
Note finally that $\exp\{- \frac{\alpha}{|I|/2}\}$ is at most $1/2$, as
otherwise the $\mu_1$-measures of both left and right halves of $I$
would be greater than $1/2$. Hence, the second factor in \eqref
{eq:second} is not greater than 1 and we have obtained the desired
estimate by $\exp\{- \frac{\alpha}{|I|/2}\}$.
\end{pf}

The next results are for the particular case of the transport in
Theorem \ref{t:sol-pb}, based essentially on the specifics of Gaussian
distributions. Namely, let $\mu_0$ and $\mu_1$ be as in Theorem \ref
{t:sol-pb}.
%

\begin{lemma}\label{l:hyp}
$\mu_0$ and $\mu_1$ satisfy the assumptions of Theorem \ref{thm:BT-line}.
\end{lemma}

\begin{pf}
The conditions (i) and (v) are obvious and the fact that the measures
$\mu_0, \mu_1$ have the same mean comes from the fact that we are
removing the same part from $\mN(0,t_0)$ and $\mN(0,1)$. Conditions
(ii) and (iii) are due to the assumptions on $\compact$. We only have
to prove (iv). Indeed, the function $\fp_\mu$ depends linearly on
$\mu
$: $\fp_{\alpha\mu+\beta\nu} = \alpha\fp_\mu+ \beta\fp_\nu$. Due
to the definition of $\mu_0$ and $\mu_1$, we have that
\begin{eqnarray*}
\operatorname{Law}\bigl(\mN(0,t_0)\bigr) &=& c\mu_0 +(1-c)
\mu',
\\
\operatorname{Law}\bigl(\mN(0,1)\bigr) &=& c\mu_1 +(1-c) \mu',
\end{eqnarray*}
where $\mu'$ is the conditional distribution law of $\mN(0,1)$ on
$\compact$.
Hence,
\begin{eqnarray*}
\fp_{\mu_0\rightarrow\mu_1}(x) = c^{-1}\fp_{\mathcal
{N}(0,t_0)\rightarrow\mathcal{N}(0,1)}(x) =
c^{-1} \int_{t_0}^1 \frac
{1}{\sqrt{2\pi t}}
\exp\biggl\{-\frac{x^2}{2t}\biggr\} \,\mathrm{d}t>0.
\end{eqnarray*}
\upqed\end{pf}

Now, let the finite expectation Brownian transport $(X_t,T_1)$, where
$T_1 = f_1(X_{T_1})$, be a continuous Brownian transport of
Theorem \ref
{thm:BT-line}. We have to show that the (continuous) function $f_1$ is
bounded. In other words, we have to estimate its behavior at infinity.
Actually, we will prove the stronger statement.
%

\begin{proposition}\label{p:maj}
$\lim_{x\to\infty} f_1(x) = 1-t_0$. Moreover, there exists a constant
$\beta>0$ such that for all $|x|$ large enough, one has $1-t_0 \le
f_1(x) \le1-t_0+e^{-\beta x^2}$.
\end{proposition}

A first step in proving this proposition is the following.
%

\begin{lemma}
$\forall x\in\R$, $f_1(x)\ge1-t_0$.
\end{lemma}

\begin{pf}
It is here easier to work with the nonnormalized measures $\hat{\mu}_0
= c\mu_0$ and $\hat{\mu}_1 = c\mu_1$, and with the corresponding
nonnormalized potential function
\[
\fp_{\hat{\mu}_0 \to\hat{\mu}_1} = c \fp_{\mu_0\to\mu_1} =
\fp_{\mN
(0,t_0)\to\mN(0,1)}.
\]
It is clear that they satisfy the system \eqref{eq:stefan}. In fact,
one can simply divide everything by $c$, to pass to the normalized
case, but it seems to us that the explanation would be less clear.

If we had not removed at the initial moment, from $\mN(0,t_0)$, the
particles corresponding to $(1-c)\mu' = \mN(0,1)|_\compact$, we would
have had
\[
\int_0^{1-t_0} \rho_{\mN(0,t+t_0)}(x) \,\mathrm{d}t
= \fp_{\hat{\mu
}_0\to\hat{\mu}_1} (x).
\]
As our initial condition is only a part of $\mN(0,t_0)$, we have
$\forall t>0\ \forall x\in\R$ $\rho_t (x) < \rho_{\mN(0,t_0)} (x)$,
where $\rho_t$ is the density of the process started with $\hat{\mu}_0$
and stopped at the moment of touching the graph of $f_1$. Hence, we
have
\[
\forall x \in\R\qquad \int_0^{t_0} \rho_t
(x) \,\mathrm{d}t < \fp_{\hat
{\mu}_0\to\hat{\mu}_1}(x),
\]
and as $\int_0^{f_1(x)} \rho_t \,\mathrm{d}t = \fp_{\hat{\mu}_0\to
\hat
{\mu}_1}$, we have proved the result.
\end{pf}

Now, let us consider the density that we obtain at the time $1-t_0$.
The next lemma estimates its behavior at infinity.

%
\begin{lemma}\label{l:encadrement}
There exists a constant $\beta_0>0$ such that for all $|x|$ large
enough, one has
\[
\rho_{\mN(0,1)}(x) \cdot\bigl(1-e^{-\beta_0 x^2}\bigr) \le
\rho_{1-t_0}(x) \le\rho_{\mN(0,1)}(x).
\]
\end{lemma}

\begin{pf}
The measure $\nu_{1-t_0}$ is the convolution of the initial measure
$\hat{\mu}_0$ with $\mN(0,1-t_0)$. If, instead of $\hat{\mu}_0$,
we had
$\mN(0,t_0)$, we would obtain exactly $\mN(0,1)$. But as $\hat{\mu}_0$
is only a part of $\mN(0,1-t_0)$, we immediately have $\rho_{1-t_0}(x)
\le\rho_{\mN(0,1)}(x)$.

The difference $\rho_{\mN(0,1)}(x) - \rho_{1-t_0}(x)$ is the part of
the density that comes from the removed part $\mN(0,1)|_\compact$ of
the initial condition. This part is supported by $[-1,1]$. Hence, the
difference
\[
\rho_{\mN(0,1)}(x) - \rho_{1-t_0}(x) = \rho_{\mN(0,1)|_\compact
\ast
\mN(0,1-t_0)} (x)
\]
can be estimated from above as $u\cdot e^{-{(x-1)^2}/{(2(1-t_0))}}$,
where $u>0$ is a constant. This is asymptotically less that $e^{-\beta
_0 x^2} \cdot\rho_{\mN(0,1)}(x)$ for any $\beta_0 < \frac{1}{2}
(\frac{1}{1-t_0} -1 )$.
\end{pf}

From now on, let us fix $\beta_0$ as in Lemma \ref{l:encadrement}. We
can estimate the behavior of the function $\fp$ at the same moment $1-t_0$.
%

\begin{lemma}\label{l:estimate}
For all $|x|$ large enough, we have $\fp_{1-t_0}(x) \le e^{-\beta_0
x^2} \rho_{\mN(0,1)}$.
\end{lemma}

\begin{pf}
From the definition of $\fp$, we indeed have
\begin{eqnarray*}
\fp_{1-t_0} (x) &=& \int_{-\infty}^x\bigl (
\mu_1\bigl((-\infty,s]\bigr) - \tilde{\nu}_{1-t_0}\bigl((-\infty,s]\bigl)
\bigr) \,\mathrm{d}s
\\
&=& \int_{-\infty}^x (x-s) (\rho_{\mu_1} -
\rho_{\tilde{\nu
}_{1-t_0}}) (s) \,\mathrm{d}s
\\
&=& \int_{-\infty}^x (x-s) (\rho_{\mN(0,1)} -
\rho_{\nu_{1-t_0}}) (s) \,\mathrm{d}s.
\end{eqnarray*}
Applying Lemma \ref{l:encadrement}, we have as $x\to-\infty$
\begin{eqnarray*}
\fp_{1-t_0}(x) &=& \int_{-\infty}^x (x-s)
\cdot e^{-\beta_0 s^2}\cdot\rho_{\mN(0,1)}(s) \,\mathrm{d}s
\\
&\le& \frac{1}{\sqrt{2\pi}} \int_{-\infty}^x |s|
e^{-(\beta_0 +1/2)s^2
} \,\mathrm{d}s = \frac{1}{\sqrt{2\pi}} \int^{\infty}_x
e^{-(\beta_0
+1/2)v^2 } \,\mathrm{d}\bigl(v^2/2\bigr)
\\
&\le& \frac{1}{\sqrt{2\pi}} e^{-(\beta_0 +1/2)x^2 } = e^{-\beta_0
x^2}\cdot
\rho_{\mN(0,1)}(x).
\end{eqnarray*}
In the same way, using the integral representation of $\fp_{\mu\to
\nu
}$ via the integral \eqref{eq:def-phi_mu}, one can estimate $\fp
_{1-t_0}(x)$ for any large positive $x$.
\end{pf}

Having obtained this estimate, we can conclude that the inequality
$f_1(x) \le1-t_0+e^{-\beta_0 x^2/2}$ will be satisfied for a ``very
dense'' at infinity set of points $x$. Namely, denote $\ell(x):=
e^{-\beta_0 x^2/2}$.

%
\begin{lemma}\label{l:points}
For any $|x|$ large enough, there exist two points $y_+\in[x,x+\ell
(x)]$ and $y_- \in[x-\ell(x),x]$ such that $f_1(y_+)\le1-t_0 +\ell
(x)$ and $f_1(y_-)\le1-t_0 +\ell(x)$.
\end{lemma}

\begin{pf}
Assume the contrary: for instance, that $\forall y\in[x,x+\ell(x)]$,
$f_1(y) > (1-t_0) +\ell(x)$. This implies that the set $\conn_t$ does
not intersect the rectangle $[x,x+\ell(x)] \times[1-t_0, 1-t_0+\ell
(x)]$, and for any point of this rectangle, the density $\rho_t (y)$
can be estimated from below via the solution of the heat equation $\dot
{u} = \frac{1}{2} \Delta u$ on $[x,x+\ell(x)]$ with the initial
conditions $u_{1-t_0} = \rho_{1-t_0}$.

For all $|x|$ large enough, $\rho_{\mN(0,1)}$ varies on $[x,x+\ell
(x)]$ at most $2$ times, and hence we have a lower bound for the
initial condition $\forall y \in[x,x+\ell(x)]$
\[
\rho_{1-t_0}(y) \ge\frac{1}{3} \rho_{\mN(0,1)} (m)\ge
\frac{1}{3} \sin\biggl( \frac{\pi}{\ell(x)} \cdot(y-x) \biggr)
\cdot
\rho_{\mN(0,1)} (m),
\]
where $m= x+\frac{1}{2}\ell(x)$ is the middle of the interval
$[x,x+\ell(x)]$.
The function $\sin(\frac{\pi}{\ell(x)} \cdot(y-x) )$ is an
eigenfunction of the Laplace operator with the eigenvalue $\lambda=
\frac{\pi^2}{\ell(x)^2}$, and hence for all $t\in[1-t_0,1-t_0+\ell
(x)]$, we have a lower bound
\begin{eqnarray*}
\rho_t (y) &\ge& \frac{1}{3} \exp\biggl\{-\frac{t-(1-t_0)}{2}
\cdot\frac{\pi
^2}{\ell(x)^2}\biggr\} \cdot\sin\biggl(\frac{\pi}{\ell(x)}
\cdot(y-x)
\biggr) \cdot\rho_{\mN(0,1)} (m)
\\
&\ge& \frac{1}{4} \sin\biggl(\frac{\pi}{\ell(x)} \cdot(y-x)
\biggr) \cdot
\rho_{\mN(0,1)} (m).
\end{eqnarray*}
In particular, for the middle point $m$ of the interval we have
\[
\rho_t (m) \ge\tfrac{1}{4} \rho_{\mN(0,1)} (m).
\]
Thus,
%
%
\begin{equation}
\label{eq:mino} \int_{1-t_0}^{1-t_0+\ell(x)} \rho_t
(m) \,\mathrm{d}t \ge\frac{\ell
(x)}{4} \cdot\rho_{\mN(0,1)} (m).
\end{equation}
As $\ell(x) = e^{-\beta_0 x^2/2}$, we have due to Lemma \ref
{l:estimate}
\[
\fp_{1-t_0} (m) \le e^{-\beta_0 (x+\ell(x))^2} \cdot\rho_{\mN
(0,1)} (m).
\]
So, we have
\[
\int_{1-t_0}^{1-t_0+\ell(x)} \rho_t (m)
\,\mathrm{d}t >\fp_{1-t_0}(m).
\]
The obtained contradiction proves the lemma.
\end{pf}

We can now complete the proof of Proposition \ref{p:maj}.

\begin{pf*}{Proof of Proposition \ref{p:maj}}
Lemma \ref{l:points} implies that for any $|x|$ large enough, either
$f_1(x) \le1-t_0 +\ell(x)$ or the connected component $I$ of $\R
\setminus\conn_t$ that contains $x$ is a subset of $[x-\ell(x),
x+\ell(x)]$. We are now going to show that then $f_1(x) \le1-t_0 +
\ell(x) +\theta_1 \ell(x) ^2$, where the constant $\theta_1$ can be
chosen not depending on $x$. Indeed, due to Lemma \ref{l:coincidence},
we can consider the continuous finite expectation Brownian transport
problem from $\nu_{1-t_0+\ell(x)}|_I$ to $\hat{\mu}_1|_I$
independently of the rest of the real line. Let us then rescale $I$ to
$[0,1]$, normalizing the measures $\nu_{1-t_0+\ell(x)}|_I$ and $\hat
{\mu}_1|_I$ to probability ones, and rescaling the time by the factor
$\frac{1}{|I|^2}$.

The density of the new probability measure $\tilde{\mu}_1$ on $\tilde
{I}=[0,1]$ takes value on $[1/2,2]$ (as $\rho_{\mN(0,1)}$ varies at
most two times on $I$). Hence, it satisfies the assumptions of
Lemma \ref{l:decomp-I} with some uniform (not depending on $x$)
constant $\alpha$. Thus, the rescaled time in which the interval
``disappears'' is uniformly (for
$|x|$ large enough) bounded by some constant $\theta_3$, and hence
$x\in I\subset\conn_{(1-t_0)+\ell(x) +\theta_3 \ell(x)^2}$. As
$\ell
(x) \ll1$, the latter statement implies the desired upper bound for $f_1(x)$.
\end{pf*}

This completes the proof of Theorem \ref{t:sol-pb}: the function $f_1$
is bounded on~$\R$.

\section{Existence of a finite expectation Brownian transport}\label
{s:BT-exists}

\subsection{Finite expectation Brownian transport on the real
line}\label{s:B-transport}\label{ss:proof-thm}
In this paragraph, we will deduce Theorem \ref{thm:BT-line} from
Theorem \ref{thm:BT-interval} (which will be proved in the next
paragraph). To do so, assume that the measures $\mu_0,\mu_1$ satisfy
the assumptions of Theorem \ref{thm:BT-line}. Naturally, the idea here
will be to find a family of compactly supported measures $\mu_0^R$ and
$\mu_1^R$ that approximate $\mu_0$ and $\mu_1$ and for which there
exist continuous finite expectation Brownian transports. The simplest
case is when the measures $\mu_0, \mu_1$, in addition to be centered
are symmetric.

We will then consider the sequence of conditional normalized measures
\[
\tilde{\mu}_0^R:= \frac{\mu_{0}|_{[-R,R]}}{\mu_0([-R,R])} \quad\mbox{and}\quad
\tilde{\mu}_1^R:= \frac{\mu_{1}|_{[-R,R]}}{\mu_1([-R,R])}.
\]
For the case of general centered measures $\mu_0$ and $\mu_1$, we will
have to modify this construction, as their restrictions on $[-R,R]$ are
no longer forced to have the same mean. Namely, denote for any measure
$\mu$ such that $\mu((-\infty,0))>0$ and $\mu((0,\infty))>0$ by
$\gamma
(\mu)$ the measure
\[
\gamma(\mu):= c(\mu) \mu|_{(-\infty, 0)} + d (\mu) \mu
|_{(0,\infty)},
\]
where $(c(\mu), d(\mu))$ is the unique solution of the system
\[
\cases{ c (\mu) \mu\bigl((-\infty, 0)\bigr) + d(\mu) \mu\bigl((0,
\infty)\bigr) = 1,\vspace*{2pt}
\cr
\displaystyle -c(\mu) \int_{-\infty}^0
|x| \,\mathrm{d}\mu+ d(\mu) \int_0^\infty x
\,\mathrm{d}\mu=0. } %
\]
It is then easy to see that $\gamma(\mu)$ is always a centered measure
and we have $c (\tilde{\mu}_j^R) \mathop{\rightarrow}\limits_{R\to\infty} 1$ and $d
(\tilde{\mu}_j^R) \mathop{\rightarrow}\limits_{R\to\infty} 1$ (as the second
equation tends to $c = d$ as $R\to\infty$). Then we can consider the
families $\mu_0^R = \gamma(\tilde{\mu}_0^R)$ and $\mu_1^R = \gamma
(\tilde{\mu}_1^R)$.

Now we would like to consider continuous finite expectation Brownian
transports from $\mu_0^R$ to $\mu_1^R$, then extract a convergent
subsequence from the sequence of corresponding functions $f_R$, and
finally show that the limit function $f$ indeed defines a continuous
finite expectation Brownian transport from $\mu_0$ to $\mu_1$.
So, a first step in the realization of this scheme is to check that for
all $R$ large enough, Theorem \ref{thm:BT-interval} is indeed
applicable for finding a continuous finite expectation Brownian
transport from $\mu_0^R$ to $\mu_1^R$.

%
\begin{lemma}\label{l:contin}
For any $R$ large enough, there exists a continuous finite expectation
Brownian transport from $\mu_0^R$ to $\mu_1^R$.
\end{lemma}

\begin{pf}
We have to check that the assumptions of Theorem \ref{thm:BT-interval}
are satisfied for all $R$ large enough. As the conditions (i)--(iii)
are the same in Theorems \ref{thm:BT-line} and \ref{thm:BT-interval},
we only have to check the two last ones.

Recall that we have $\lambda:= \limsup_{x\to\infty} \frac{\rho
_{\mu
_0}(x)}{\rho_{\mu_1}(x)}<1$. Hence, for some constant $M$, we have
$\frac{\rho_{\mu_0}(x)}{\rho_{\mu_1}(x)}< \frac{1+\lambda}{2}$ outside
$[-M,M]$. Now, for $x\in(-M,M)$, we have
\[
\frac{\rho_{\mu_0^R}(x)}{\rho_{\mu_1^R}(x)} = \frac{\rho_{\mu
_0}(x)}{\rho_{\mu_1}(x)} \cdot\frac{\mu_1 ([-R,R])}{\mu_0
([-R,R])} \cdot\biggl(
\frac{c(\tilde{\mu}_0^R)}{c(\tilde{\mu}_1^R)}\cdot\1 _{x<0} +
\frac{d(\tilde{\mu}_0^R)}{d(\tilde{\mu}_1^R)}\cdot
\1_{x\ge0} \biggr).
\]
Note that the second factor in the right-hand side tends (uniformly) to
1 as $R\to\infty$. Thus, for any $R$ large enough, it is less than
$\frac{2}{1+\lambda}$, and hence $\exists M\dvtx\forall|x|>M, \frac
{\rho_{\mu_0^R}(x)}{\rho_{\mu_1^R}(x)} < \frac{2}{1+\lambda}
\cdot\frac
{1+\lambda}{2} = 1$. This proves the desired condition (v).

Moreover, note that due to the finiteness of the first moment of $\mu
_0$ and $\mu_1$, we have $\fp_{\mu_0^R \to\mu_1^R}(x)
\mathop{\rightarrow}\limits_{R\to
\infty
} \fp_{\mu_0 \to\mu_1}(x)$ uniformly on $x\in[-M,M]$. Thus, for all
$R$ large enough, we have $\fp_{\mu_0^R \to\mu_1^R}>0$ on $[-M,M]$.

Next, for all $R>M$ and $x\in(-R,-M]$, we have
\begin{eqnarray*}
\fp_{\mu_0^R \to\mu_1^R} (x)& = &\int_{-\infty}^x \bigl(
\mu_1^R - \mu_0^R\bigr) \bigl((-
\infty,s]\bigr) \,\mathrm{d}s \\
&= &\int_{-\infty}^x (x-s) \bigl(
\rho_{\mu
_1^R}(s) - \rho_{\mu_0^R}(s)\bigr) \,\mathrm{d}s >0.
\end{eqnarray*}
Finally, if $R>M$ and $x\in[M,R)$, we have
\begin{eqnarray*}
\fp_{\mu_0^R \to\mu_1^R} (x) &=& \int^{\infty}_x \bigl(
\mu_1^R - \mu_0^R\bigr)
\bigl([s,+\infty)\bigr) \,\mathrm{d}s \\
&=& \int^{\infty}_x
(s-x) \bigl(\rho_{\mu
_1^R}(s) - \rho_{\mu_0^R}(s)\bigr) \,\mathrm{d}s
>0.
\end{eqnarray*}
Thus, for all $R$ large enough and all $x\in(-R,R)$, we have $\fp
_{\mu
_0^R \to\mu_1^R} (x)>0$. This proves (iv), and thus completes the proof.
\end{pf}

We will choose and fix a value $R_0\ge1$ such that for any $R>R_0$,
there exists a continuous finite expectation Brownian transport from
$\mu_0^R$ to $\mu_1^R$, and we will consider the corresponding family
of stopping functions $f_R$.

A next step is to assure the possibility of extracting a convergent
subsequence from the family of functions $f_R$.
%

\begin{proposition}\label{p:pre-compact}
The family $(f_R)$ is precompact in the topology of uniform convergence
on the compact sets.
\end{proposition}

This proposition, due to the Arzel\`a--Ascoli theorem, is equivalent to
the union of the following two results.

%
\begin{lemma}\label{l:lip}
The family of functions $(f_R)$ is locally uniformly bounded: for any
interval $I= [-\ell,\ell]$, there exists $C'=C'(\ell)$ such that
$\forall R\ge R_0$, we have $f_R|_{I} \le C'$.
\end{lemma}

%
\begin{proposition}\label{p:equicont}
Let $\mu_0,\mu_1$ be two probability measures, supported on a finite or
infinite interval $I\subset\R$,
for which there exists a continuous finite expectation Brownian
transport from $\mu_0$ to $\mu_1$ with some stopping function $f$.
Assume that, for an interval $I' \subset I$ and a constant $C'>0$, the
following holds:
\begin{longlist}[(iii)]
\item[(i)]$\mu_0, \mu_1$ satisfy the hypotheses of Theorem \ref
{thm:BT-interval} on $\mathcal{U}_1(I')\cap I$.
\item[(ii)]$f|_{\mathcal{U}_1(I')\cap I} \le C'$.
\item[(iii)]$\mu_0|_{I'}$ and $\mu_1|_{I'}$ satisfy the hypotheses of
Theorem \ref{thm:BT-interval} for some constants $a',b',\alpha'$.
\end{longlist}
Let $\delta_0:= \min\{\frac{\varepsilon}{3\theta_0 \alpha'},
\frac
{1}{2}\}$.
Then the inverse of the modulus of continuity of $f|_{I'}$, denoted by
$\delta_{f|_{I'}}(\varepsilon)$, is lower bounded by
%
%
\begin{equation}
\label{eq:modulus} \delta_{f|_{I'}} (\varepsilon) \ge\frac
{\varepsilon\pi\cdot
a'}{2\delta_0 \cdot b'} \exp
\biggl\{-\frac{\pi^2 C'}{\delta_0^2}\biggr\}.
\end{equation}
\end{proposition}

\begin{pf*}{Proof of Lemma \ref{l:lip}}
We will first prove that the functions $f_R$ ``take small values
somewhere.'' Namely, that there exist some constants $\ell_1,C''$ such
that $\forall R\ge R_0$, $\exists x\in[-\ell_1, \ell_1 ]$: $f_R(x)
\le
C''$. Indeed, as we have already mentioned, the functions $\fp^R:=
\fp
_{\mu_0^R \rightarrow\mu_1^R}$ converge to the function $\fp:= \fp
_{\mu_0 \rightarrow\mu_1}$. In particular, the values $\fp^R (0)$ are
uniformly bounded by some constant $C_1$.

Now, let us consider a Brownian motion started from $\mu_0|_{[-1,1]}$.
Its density $\rho_{\mathrm{BM}}$ at 0 has an asymptotics of $\frac{1}{\sqrt
{t}}$, and thus, its integral diverges. Hence, there exists $C''$ such that
%
%
\begin{equation}
\label{eq:min-BM} \int_0^{C''} \rho_{\mathrm{BM}}(t)
\,\mathrm{d}t > C_1.
\end{equation}
By continuity, \eqref{eq:min-BM} holds also in the case of the density
$\rho$ of the process starting with an initial measure $\mu
_0^R|_{[-1,1]} > \mu_0|_{[-1,1]} $, and which trajectories are stopped
outside a large enough interval $[-\ell_1,\ell_1]$. Hence, for any $R$
large enough (so that $\mu_0^R|_{[-1,1]}$ is close enough to $\mu
_0|_{[-1,1]}$), there exists $x\in[-\ell_1,\ell_1]$ such that $f(x)
\le C''$. Indeed, otherwise, we would have an inequality
\[
\int_0^{C''} \rho_t^R(0)
\,\mathrm{d}t > \fp_{\mu_0^R \rightarrow\mu
_1^R} (0),
\]
which would be a contradiction.

Now, for the finite expectation Brownian transport from $\mu_0^R$ to
$\mu_1^R$, let us consider the total measure $\nu_t (\R\setminus
\conn
_t)$ of the not yet stopped trajectories at some time $t$. Note that,
due to the recurrence of the Brownian motion on $\R$: $\forall
\varepsilon>0$, $\forall\ell_2$, there exists a time $\bar{t}= \bar
{t}(\varepsilon, \ell_2)$ such that for any $x\in[-\ell_2,\ell
_2]$, a
Brownian trajectory, starting at $x$, crosses the rectangle $[-\ell
_1,\ell_1] \times[C'',\bar{t}]$ left to right with probability at
least $1-\varepsilon$.

Choose now $\ell_2$ large enough so that $\forall R\ge R_0$, $\mu_0^R
([-\ell_2,\ell_2]) \ge1-\varepsilon$. Then, for any $R\ge R_0$, the
total measure $\nu_{\bar{t}}(\R\setminus\conn_{\bar{t}})$ of the not
yet stopped trajectories at time $\bar{t}$ will be at most
$2\varepsilon
$, as crossing the rectangle implies stopping due to the choice of
$\ell
_1$ and $C''$. In particular, taking
\[
\varepsilon:= \tfrac{1}{4} \min\bigl(\mu_0(-\ell-1, -\ell),
\mu_0(\ell,\ell+1)\bigr),
\]
we see that
\begin{eqnarray*}
\nu_{\bar{t}} (\R\setminus\conn_t) &\le& \tfrac{1}{2}
\mu_0(-\ell-1,-\ell) \le\mu_0^R (-\ell-1, -
\ell),
\\
\nu_{\bar{t}} (\R\setminus\conn_t) &\le& \tfrac{1}{2}
\mu_0(\ell,\ell+1) \le\mu_0^R (\ell, \ell+1).
\end{eqnarray*}
Hence, any connected component of $\R\setminus\conn_{\bar{t}}$ that
intersects $I=(-\ell, \ell)$ is contained in $(-\ell-1,\ell+1)$.

Applying now Lemma \ref{l:decomp-I} for all the connected components of
$\R\setminus\conn_{\bar{t}}$ that intersect $I$, we conclude that all
of them disappear in at most time $\theta\cdot\alpha_{\ell+1} \cdot
|[-\ell-1, \ell+1]|$. Hence, $\forall R\ge R_0$, $f_R|_{[-\ell, \ell]}
\le\bar{t} + \theta\cdot\alpha_{\ell+1} \cdot(2\ell+2)$ and we
have the desired upper bound.
\end{pf*}

We are now ready to prove the uniform continuity for the family $f_R$,
that is, Proposition \ref{p:equicont}. A basic idea here is the
following one: assume that the function $f$ is smooth and (piecewise)
monotonic. Then, considering a point $x$ in a neighborhood of which $f$
is monotonically increasing, we see that between the moments $t=f(x)$
and $t+\Delta t = f(x+\Delta x)$, the left end of the interval of
complement to $\conn_t$ absorbs approximatively the mass $\Delta t
\cdot\rho_t'(x)$ of Brownian particles and this should be equal to the
mass $\mu_1$ of the interval $[x,x+\Delta x]$. Hence,
\[
\Delta t \approx\frac{\mu_1([x,x+\Delta x])}{\rho_t'(x) } \approx
\frac
{\rho_{\mu_1} (x) }{\rho_t'(x)} \cdot\Delta x.
\]
Estimating from above the numerator by $b$, and from below the
denominator (by a comparison with the heat equation on an interval), we
obtain the desired bound for $f' = \frac{\Delta t}{\Delta x}$. Let us
now make these computations rigorous.

\begin{pf*}{Proof of Proposition \ref{p:equicont}}
Note first that Lemma \ref{l:decomp-I} guarantees that the functions
$f|_I$ cannot have ``high thin peaks'': if $y,z\in\mathcal
{U}_1(I')\cap I$ and $f(y) = f(z)$, then
\[
\max_{x\in[y,z]} f(x) \le f(y) + \theta\alpha'
\cdot\bigl|[y,z]\bigr|.
\]
Now, take $\delta= \delta_0 = \min(\frac{\varepsilon}{\theta_0
\alpha'}, \frac{1}{2} )$ and let us show the estimate \eqref
{eq:modulus}. Namely, assume first that $x,y\in I'$ with the distance
between $x$ and $y$ less than the right-hand side of \eqref
{eq:modulus}. We want to show that $|f(x) - f(y)| \le\varepsilon$.
Without any loss of generality, we can assume that $f(x) <f(y)$. We can
also assume that $\forall x' \in[x,y]$, $f(x') >f(x)$ (as otherwise,
we can replace $x$ with the rightmost point $x'$ of the level set
$f^{-1}(f(x)) \cap[x,y]$).

Consider now the behavior of $f$ on $[x,x+\delta_0]$. Denote $t_1
=f(x)$ and $t_2= \min_{[y, x+\delta_0]} f$. Due to Lemma \ref
{l:decomp-I} and the choice of $\delta_0$, we have
\[
f(y) \le\max(t_1,t_2) + \theta\alpha'
\delta_0 \le\max(t_1,t_2) +
\frac{\varepsilon}{2}.
\]
Thus, if $t_2\le t_1 + \frac{\varepsilon}{2}$, everything is proven.
(In particular, this rules out the case of $x+\delta_0$ falling outside
$I$: the lower limit of $f$ at an endpoint of $I$ is zero.)

Thus, we can assume that $t_2>t_1 + \frac{\varepsilon}{2}$. Consider
now the Brownian paths of the process $X_t$ that were not stopped. Note
that any such path, starting anywhere in $[x,x+\delta_0]$, stays in
this interval until the moment $t_1$ and then leaves it through the
left end before the moment $t_2$, as shown in Figure \ref{fig:2} below.
The first intersection point of such a path with the graph of $f$ is
somewhere above $[x,y]$. Hence, the measure $\mu_1([x,y])$ is greater
or equal to the measure of such paths.

%
\begin{figure}[b]

\includegraphics{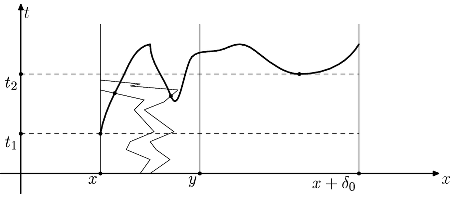}
\caption{Two Brownian paths crossing the graph of $f$.}\label{fig:2}
\end{figure}

Finally, we can easily estimate this measure from below through the
heat equation. Namely, the condition $\rho_{\mu_0}|_I \ge a$ allows us
to estimate the initial density on $[x,x+\delta]$ from below by an
eigenfunction of the Laplace operator, that is $u(z) = a \sin\frac
{\pi
(z-x)}{\delta_0}$ with the eigenvalue $\lambda= \frac{\pi^2}{\delta
_0^2}$. Hence,\vspace*{-2pt} the density of the trajectories that have never left
$[x,x+\delta]$ up to time $t$ is greater than $e^{-\lambda t} \cdot a'
\sin\frac{\pi(z-x)}{\delta_0}$, and thus the density of those who are
first-leaving the interval through its left end is at least $a' \frac
{\pi}{\delta_0} e^{-\lambda t}$. The total mass of the trajectories
leaving between the moments $t_1$ and $t_2$ is
\[
\int_{t_1}^{t_2} a \frac{\pi}{\delta_0}
e^{-\lambda t} \,\mathrm{d}t \ge a' \frac{\pi}{\delta_0}
(t_2-t_1) e^{-\lambda t_2}.
\]
As we have $t_2-t_1 \ge\frac{\varepsilon}{2}$ and $t_2\le C'$, we
finally have obtained a lower bound for the total mass of such
trajectories and thus for $\mu_1([x,y])$. This lower bound is given by
\[
a' \frac{\pi\cdot\varepsilon}{\delta_0 \cdot2} e^{-C' \pi
^2/\delta_0^2}.
\]
Though, due to our assumption, $\mu_1([x,y]) \le b' (y-x)$, and due to
our choice of $\delta(\varepsilon)$, this gives us a contradiction.
\end{pf*}

Having proved both Lemma \ref{l:lip} and Proposition \ref{p:equicont},
we have thus proved Proposition \ref{p:pre-compact}. We are now ready
to start concluding the proof of Theorem \ref{thm:BT-line}. Namely, as
the family $(f_R)$ is precompact, there exists a convergent subsequence
$f_{R_k} \mathop{\rightarrow}\limits_{k\to\infty} f$. A natural conclusion would
then be that the first intersection measure with the graph of $f$ for
the initial measure $\mu_0 = \lim\mu_0^{R_k}$ is exactly $\mu_1 =
\lim
\mu_1^{R_k}$. To make this argument work rigorously, we will need the
following.
%

\begin{definition}
Let $f\in C(\R,\R_+)$ be a continuous positive function and $x\in\R$.
The \emph{first intersection measure} $m_{x,f}$ is defined as the law
of the $x$-coordinate of the first intersection between the graph of
$f$ and the trajectory of the Brownian motion started from the point
$x$: $X_t= x+B_t$, $T=\inf\{t\ge0\dvtx t=f(X_t)\}$ and $m_{x,f}=
\operatorname
{Law}(X_T)$. Similarly, we denote by $m_{\mu,f}$ the first intersection
measure between the process started from the distribution $\mu$ and the
graph of the stopping function $f$.
\end{definition}

%
\begin{proposition}\label{p:intersection}
The first intersection measure $m_{x,f}$ depends continuously (in the
sense of the weak* convergence) on $x\in\R$ and $f\in C(\R,\R_+)$
[where $C(\R,\R_+)$ is equipped with the topology of uniform
convergence on compact sets].
\end{proposition}

The following lemma is an easy exercise.

%
%
\begin{lemma}\label{l:bm-cross}
Denote by $(X_t,t\ge0)$ the standard Brownian motion. For all
$\varepsilon>0$, there exists $\delta>0$ such that, with probability at
least $1-\varepsilon$, there exist $t_+,t_-\in[\delta, \varepsilon]$
such that $X_{t_+}=\delta$, $X_{t_-} = -\delta$ and $\sup_{0\le t\le
\max(t_+,t_-)} |X_t|\le\varepsilon$. In other words, the Brownian
motion crosses horizontally the rectangle $[-\delta,\delta]\times
[\delta,\varepsilon]$, and before this crossing, it stays inside the
strip $[-\varepsilon,\varepsilon]\times\R_+$ (see Figure~\ref
{fig:3}).
\end{lemma}

\begin{pf*}{Proof of Proposition \ref{p:intersection}}
Let $f_1\in C(\R,\R_+)$ and $x_1\in\R$ be given. Take an arbitrary
$\varepsilon>0$ and let $\delta>0$ be defined by Lemma \ref
{l:bm-cross}. It is easy to see that, for some $R>0$, for any initial
point $x\in\mathcal{U}_1(x_1)$ and for any $f$ such that
$|f(x_1)-f_1(x_1)| \le1$, the Brownian motion started at $x$
intersects $f$ before leaving the strip $[-R,R]\times\R_+$ with
probability at least $1-\varepsilon$.

%
\begin{figure}

\includegraphics{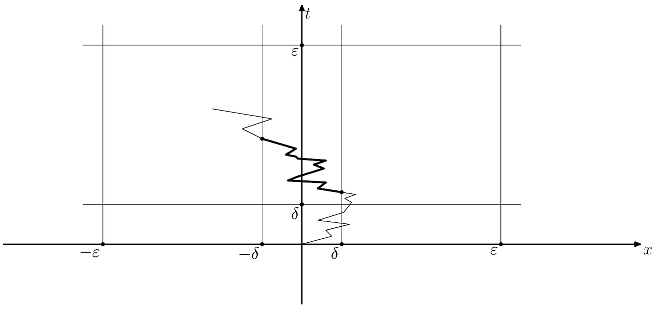}

\caption{A Brownian path crossing the strip.}\label{fig:3}
\end{figure}

Consider now $x_2\in\mathcal{U}_\delta(x_1)$ and $\|f_2-f_1\|
_{C([-R-\delta, R+\delta])} \le\delta$. We will estimate the
difference between $m_{x_1,f_1}$ and $m_{x_2,f_2}$. To do this, take
the trajectory of the same Brownian motion $B_t$ shifted to the initial
points $x_1$ and $x_2$: $X_t^1 = x_1+B_t$ and $X_t^2 = x_2+B_t$.

Consider the moment of the first intersection of these processes with
the corresponding graphs. Let $T_j:= \inf\{t\ge0\dvtx t=f_j(X_t^j)\}$
for $j=1,2$ and $T:=\min(T_1,T_2)$. Note that $T_1$ and $T_2$ are two
Markov hitting times and hence, the conditional behavior of $X_t^j$
under any condition $T=T_0$ and $X_{T_0}^j = \bar{x}_j$ is simply the
Brownian motion shifted to the initial point $(T_0, \bar{x}_j)$. See
Figure \ref{fig:4} below.

%
\begin{figure}[t]

\includegraphics{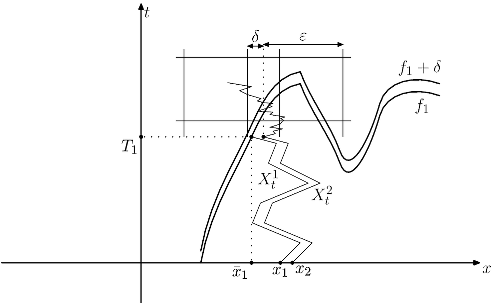}

\caption{First intersection of $X^1,X^2$ with the graphs of $f_1$ and
$f_1+\delta$, respectively.}\label{fig:4}
\end{figure}

Now, let us prove that we have $|X_{T_1}^1 -X_{T_2}^2|\le\varepsilon$
with probability at least $(1-\varepsilon)^2$. To show this, we first
note that, due to the choice of $R$, we have $X_T^j \in\mathcal
{U}_{R+\delta}(x_1)$ with probability at least $1-\varepsilon$. Now,
under any ``first intersection condition'' $T_2\ge T_1=\bar{t}$,
$X_{\bar{t}}^1 = \bar{x}_1\in\mathcal{U}_R(x_1)$, the trajectory of
$X_{\bar{t}}^2$ intersects the graph of $f_2$ inside $\mathcal
{U}_{\varepsilon+\delta}(\bar{x}_1)\times[\bar{t}, \bar
{t}+\varepsilon
]$ with probability at least $1-\varepsilon$. Indeed, under this
condition, the trajectory of $X_{\bar{t}}^2$ is the trajectory of the
Brownian motion started from the point $(\bar{t}, X_{\bar{t}}^2)$.
Meanwhile, we have $|X_{\bar{t}}^2 -X_{\bar{t}}^1| =|x_2-x_1|\le
\delta
$. Also, we have $f_2(\bar{x}_1)\le f_1(\bar{x}_1)+\delta$. Recalling
the definition of $\delta$, we obtain the desired estimate on the
conditional probability.

In the same way, under any condition $T_1\ge T_2=\bar{t}$ and $X_{\bar
{t}}^2 = \bar{x}_2\in\mathcal{U}_R(x_1)$, we have $|X_{T_1}^1
-X_{T_2}^2|\le\varepsilon+ \delta$ with probability at least
$1-\varepsilon$. Considering the first intersection moment, we see
that, with probability at least $(1-\varepsilon)$, the corresponding
point belongs to $\mathcal{U}_R(x_1)$, and conditionally to it we have
$|X_{T_1}^1-X_{T_2}^2| \le\varepsilon+\delta$ with probability at
least $1-\varepsilon$. Hence, we have finally
%
%
\begin{equation}
\label{eq:min-proba} \bbP\bigl(\bigl|X_{T_1}^1 -X_{T_2}^2\bigr|
\le\varepsilon+\delta\bigr) \ge(1 - \varepsilon)^2.
\end{equation}

As $m_{x_1,f_1} =$ Law$(X_{T_1}^1)$ and $m_{x_2,f_2} =$
Law$(X_{T_2}^2)$, \eqref{eq:min-proba} gives us the desired comparison
between these two measures.
\end{pf*}

As it can be easily seen from the latter proof, the continuity in
Proposition \ref{p:intersection} is uniform for $x$ belonging to any
compact set in $\R$.

For further arguments, it will be useful to consider the following
distance between probability measures.
%

\begin{definition}
Let $\mu,\mu'$ be two probability measures. We define the \emph{Prohorov
distance} between them as
$d(\mu, \mu'):= \inf\{\delta>0\dvtx\exists$ random
variables $U,V \dvtx \operatorname{ Law}(U) = \mu, \operatorname{ Law}(V) = \mu'
\mbox{ and }
\bbP
(|U-V|\le\delta) \ge1-\delta\}$.
\end{definition}

%
\begin{remark*}
The Prohorov distance between two probability measures $\mu,\mu'$ is
usually defined as
\[
\bigl|\mu- \mu'\bigr|_P:= \inf\bigl\{\delta>0\dvtx\mu(A) \le
\mu'\bigl(\mathcal{U}_\delta(A)\bigr) + \delta\ \forall A
\in\mathcal{B}(\R)\bigr\},
\]
where $\mathcal{U}_\delta(A)$ is the $\delta$-neighborhood of $A$. But
Strassen's theorem (see \cite{strassen}) proves that these distances
are equivalent.
\end{remark*}

It is easy to see that this distance defines on the space of
probability measures precisely the weak* convergence. In fact, in the
proof of Proposition \ref{p:intersection}, we obtain the estimate
\[
d(m_{x_1,f_1}, m_{x_2,f_2}) \le\max\bigl(1-(1-\varepsilon)^2,
\varepsilon+\delta\bigr) \le2\varepsilon.
\]
Now, let us pass to the first intersection measures starting from
arbitrary initial distributions.

%
\begin{lemma}\label{l:conv}
Let $\mu_0^{(k)}\rightarrow\mu_0$ be a weak* convergent sequence of
measures, and $f_{(k)},f \in C(\R,\R_+)$ be such that
$f_{(k)}\rightarrow f$ uniformly on any compact set. Then $m_{\mu
_0^{(k)}, f_{(k)}} \mathop{\rightarrow}\limits_{k\to\infty} m_{\mu_0,f}$.

If, additionally, the corresponding expectations of the first
intersection times $T_{(k)}$ are uniformly bounded by some constant
$C$, then the expectation of the first intersection time $T$ is also
finite and does not exceed $C$.
\end{lemma}

\begin{pf}
Indeed, for any $\varepsilon>0$, there exist $\ell_1,\ell_2, \delta
>0, \delta\le\varepsilon$ such that:
\begin{longlist}[(ii)]
\item[(i)]$\mu_0(-\ell_1,\ell_1) \ge1-\varepsilon$,
\item[(ii)] if $|x|\le\ell_1$, $|y-x| \le\delta$ and $\|f- \tilde{f}\|
_{C([-\ell_1 -\ell_2, \ell_1+\ell_2])} \le\delta$, then we have
$d(m_{x,f}, m_{y,\tilde{f}})\le\varepsilon$.
\end{longlist}
(The second conclusion comes from the uniform version of
Proposition \ref{p:intersection}.)

For any $k$ large enough, we have $d(\mu_0, \mu_0^{(k)}) <\delta$.
Hence, for any such $k$, we can choose the processes $X^1, X^2$ such
that Law$(X_0^1)=\mu_0$, Law$(X_0^2) = \mu_0^{(k)}$, $\mathrm
{d}X_t^1 =
\mathrm{d}X_t^2 = \mathrm{d}B_t$ and $\bbP(|X_0^1-X_0^2| \le\delta)
\ge1-\delta$. Then, with probability at least $1-\delta-\varepsilon$,
we have
%
%
\begin{equation}
\label{eq:maj2} \bigl|X_0^1\bigr| \le\ell_1 \quad\mbox{and}\quad
\bigl|X_0^1 -X_0^2\bigr|\le\delta.
\end{equation}
Due to the property (ii), the conditional probability of $|X_{T}^1 -
X_{T_k}^2|\le\varepsilon$ is at least $1-\varepsilon$ under the
condition \eqref{eq:maj2}, where $T = \inf\{t\ge0\dvtx t=f(X_t^1)\}$
and $T_{(k)} = \inf\{t\ge0\dvtx t=f_{(k)}(X_t^2)\}$ are first
intersection stopping times.

Hence, with probability at least $1-\delta-2\varepsilon$, we have
$|X_{T}^1 - X_{T_2}^2|\le\varepsilon$, and hence
\[
d(m_{\mu_0,f}, m_{\mu_0^{(k)},f_{(k)}}) \le\delta+2\varepsilon\le
3\varepsilon.
\]
As $\varepsilon$ is arbitrarily chosen, we have $m_{\mu
_0^{(k)},f_{(k)}} \mathop{\rightarrow}\limits_{k\to\infty} m_{\mu_0,f}$.

Now, let us prove the second statement of the lemma. Actually, for any
$k$ large enough, and any realization as before, we have
$|T-T_{(k)}|\le\varepsilon$ with probability at least $1-\delta
-2\varepsilon\ge1-3\varepsilon$. Thus, we have obtained a lower bound
for the integral of $T$ over a set of probability $1-3\varepsilon$,
which is $\E T_{(k)} +\varepsilon\le C+\varepsilon$. As $\varepsilon
>0$ is arbitrary, this implies that $\E T \le C$.
\end{pf}

We can now conclude the proof of Theorem \ref{thm:BT-line}.

\begin{pf*}{Proof of Theorem \ref{thm:BT-line}}
We have now constructed continuous finite expectation Brownian
transports from $\mu_0^{R_k}$ to $\mu_1^{R_k}$ with stopping functions
$f_{R_k}$ converging uniformly on compact sets to some continuous
function $f$. Then, due to the first part of Lemma \ref{l:conv}, we
have
\[
m_{\mu_0,f} = \lim_{k\to\infty} m_{\mu_0^{R_k}, f_{R_k}} = \lim
_{k\to
\infty} \mu_1^{R_k} = \mu_1.
\]
The expectations of the corresponding passage times $T_{(k)}$ are also
equal to
\[
\E T_{(k)} = \Var\mu_1^{R_k} - \Var
\mu_0^{R_k}
\]
and thus, due to the choice of $\mu_0^{R_k}, \mu_1^{R_k}$, the latter
difference converges to $\Var\mu_1 - \Var\mu_0<\infty$. Hence, these
expectations are uniformly bounded and due to the second part of
Lemma \ref{l:conv}, we have $\E T<\infty$. We have finally constructed
a continuous finite expectation Brownian transport from $\mu_0$ to
$\mu_1$.
\end{pf*}

\subsection{Finite expectation Brownian transport on an interval:
Discretization}\label{s:B-transport-interval}

\subsubsection{Discretization}\label{ss:discretization}

We are now going prove Theorem \ref{thm:BT-interval}. As we have
already mentioned, we will do it by means of a discretization
procedure, replacing the Brownian motion by a discrete random walk, and
then passing to the limit as the mesh of the lattice goes to zero.

We will first study a discretized version of our problem. Namely,
instead of a Brownian motion on $\R$, we consider a random walk on $\Z
$:
\[
Y_{t+1} = %
\cases{ Y_t +1, &\quad $\mbox{with
probability } 1/2$,\vspace*{2pt}
\cr
Y_t -1, & \quad $\mbox{with
probability } 1/2$.} %
\]
We have to modify the setting of a continuous finite expectation
Brownian transport in the following way. The stopping time $T$ is now a
probabilistic Markov moment, that is related to the new function $g$ in
the following way:
%
%
\begin{equation}
\label{e:def-g} %
\cases{ \mbox{if } t >g(Y_t), \mbox{ then
the process is stopped},\vspace*{2pt}
\cr
\mbox{if } t=g(Y_t), \mbox{
then the process is stopped with probability $q(Y_t)$},} %
\end{equation}
where $q\dvtx \Z\rightarrow[0,1]$ is a new auxiliary function. A finite
expectation Brownian transport in this setting will be called a \textit{discrete Brownian transport}.

The new discrete functions corresponding to $\fp$ are defined as
\[
\fp_\mu^\Z(x) = \sum_{y<x}
\sum_{z\le y} \mu(z) = \sum
_{z<x} (x-z)\mu(z),
\]
and $\fp_{\mu_0\rightarrow\mu_1}^{\Z} (x):= \fp_{\mu_1}^{\Z}
(x) - \fp
_{\mu_0}^{\Z} (x)$. It is then easy to check that for a centered
measure $\mu$ on $\Z$ and for an integer $x$, one has $\fp_\mu(x) =
\fp
_\mu^{\Z}(x)$. So, we will in further mostly omit the upper index
``$\Z
$.'' The discrete function $\fp$ works in the same way as its
continuous analogue: an easy computation shows that
\[
\fp_{\delta_0 \rightarrow({1}/{2})(\delta_{-1} + \delta_1)} (x)
= \tfrac{1}{2} \delta_0(x).
\]
Hence, we have for any displacement defined by \eqref{e:def-g}
%
%
\begin{equation}
\label{eq:rec} \fp_{\nu_t \rightarrow\nu_{t+1}} (x) = \frac{1}{2}
\cdot%
\cases{ \nu_t(x), & \quad $\mbox{if } g(x) >t$,\vspace*{2pt}
\cr
0, & \quad $
\mbox{if } g(x)<t$,\vspace*{2pt}
\cr
\nu_t(x) \cdot q(x), &\quad  $\mbox{if } g(x)=t$.} %
\end{equation}

This allows us, for two centered measures $\mu_0,\mu_1$, to define
recursively the transport process in the following way:
\begin{longlist}[(ii)]
\item[(i)] Initial state: $K_{-1}=\varnothing$.
\item[(ii)] Evolution: for any $t\ge0$, any $x\in\Z\setminus K_{t-1}$, if
$\fp_{\nu_t \rightarrow\mu_1} (x) > \frac{1}{2} \nu_t(x)$, where
$\nu
_t$ is the occupation measure at time $t$, there is nothing to be done.
Otherwise, take $g(x):= t$ with $q(x) = 2\frac{\fp_{\nu_t
\rightarrow
\mu_1}(x)}{\nu_t(x)}$ [and 0 if $\fp_{\nu_t \rightarrow\mu
_1}(x)=\nu_t(x)=0$].
\end{longlist}
Due to \eqref{eq:rec}, we then have
\[
\fp_{\nu_{t+1} \rightarrow\mu_1}(x) = \fp_{\nu_t \rightarrow\mu
_1}(x) - \min\bigl(\tfrac{1}{2}
\nu_t(x), \fp_{\nu_t \rightarrow\mu_1}(x) \bigr).
\]
In particular, we can easily see by induction that all the functions
$\fp_t:= \fp_{\nu_t \rightarrow\mu_1}$ are nonnegative, and the
procedure is thus well defined for all $t$. Also, the latter
construction implies the following:
\begin{longlist}[(iii)]
\item[(i)] if at some time $t$, at cell $x$, we have $\fp_{\nu_t
\rightarrow
\mu_1}(x) =0$, then the cell $(t,x)$ is frozen and any particle coming
to it at this moment (or afterward) is stopped,
\item[(ii)] if $\fp_{\nu_t \rightarrow\mu_1}(x) \ge\frac{1}{2}\nu_t (x)$,
then the cell $(t,x)$ is fully diffused,
\item[(iii)] if $0<\fp_{\nu_t \rightarrow\mu_1}(x)< \frac{1}{2} \nu_t(x)$,
then the cell $(t,x)$ is ``partially frozen,'' meaning that a part of
the particles of total measure $2\fp_{\nu_t \rightarrow\mu_1}(x)$ is
diffused, whereas the others are frozen. In this case, $\fp_{\nu_{t+1}
\rightarrow\mu_1}(x) =0$, so that, starting from the moment $t+1$, the
cell $x$ becomes fully frozen.
\end{longlist}

We have the following.
%

\begin{proposition}\label{p:discret}
Let $\mu_0,\mu_1$ be two centered measures on $\Z$, both with finite
support. Suppose that $\mu_1$ is everywhere positive on the interval
$I:= [\min\operatorname{ Supp} (\mu_0), \max\operatorname{ Supp} (\mu_0)]$ and
$\fp_{\mu_0 \rightarrow\mu_1}\ge0$. Then the procedure \eqref
{e:def-g} provides us with everywhere defined functions $g,q$ that
define a discrete bounded Brownian transport from $\mu_0$ to $\mu_1$.
\end{proposition}

To prove this result, we will first need the following lemma, which is
a discrete analogue of Lemma \ref{l:coincidence}.

%
\begin{lemma}\label{l:equality-discrete}
Let $\mu, \nu$ be two centered (discrete) measures of finite support.
Suppose that $\fp_{\nu\rightarrow\mu} \ge0$ and $\fp_{\nu
\rightarrow\mu} (x) = \fp_{\nu\rightarrow\mu} (y) =0$ for some
$x<y$. Then, we have $\mu([x,y]) \ge\nu([x,y]) \ge\nu([x+1,y-1])
\ge\mu([x+1,y-1])$.
\end{lemma}

\begin{pf}
Note that $\nu(z) = (\fp_\nu(z+1)-\fp_\nu(z) ) - (\fp
_\nu(z)-\fp_\nu(z-1) )$. Taking the difference between such
representations for $\mu(z)$ and $\nu(z)$, and summing up on $z\in
[x+1,y-1]$, we have
\begin{eqnarray*}
&&\sum_{z\in[x+1,y-1]} \bigl(\mu(z) - \nu(z) \bigr) \\
&&\qquad= \bigl(
\fp_{\nu
\rightarrow\mu} (y)-\fp_{\nu\rightarrow\mu}(y-1) \bigr) - \bigl
(\fp
_{\nu\rightarrow\mu} (x+1)-\fp_{\nu\rightarrow\mu}(x) \bigr)
\\
&&\qquad= - \fp_{\nu\rightarrow\mu}(y-1) - \fp_{\nu\rightarrow\mu}(x+1).
\end{eqnarray*}
Hence, we get
\[
\nu\bigl([x+1,y-1]\bigr)- \mu\bigl([x+1,y-1]\bigr) = \fp_{\nu
\rightarrow\mu}(y-1) +
\fp_{\nu\rightarrow\mu}(x+1).
\]
On the other hand, summing on $z\in[x,y]$, we have
\begin{eqnarray*}
&&\sum_{z\in[x,y]} \bigl(\mu(z) - \nu(z) \bigr)\\
&&\qquad= \bigl(
\fp_{\nu
\rightarrow\mu} (y+1)-\fp_{\nu\rightarrow\mu}(y) \bigr) - \bigl
(\fp
_{\nu\rightarrow\mu} (x)-\fp_{\nu\rightarrow\mu}(x-1) \bigr)
\\
&&\qquad= \fp_{\nu\rightarrow\mu}(y+1) + \fp_{\nu\rightarrow\mu
}(x-1)\ge0.
\end{eqnarray*}
Thus, we conclude that
\[
\mu\bigl([x,y]\bigr) \ge\nu\bigl([x,y]\bigr) \ge\nu\bigl
([x+1,y-1]\bigr) \ge\mu
\bigl([x+1,y-1]\bigr).
\]
\upqed\end{pf}

\begin{pf*}{Proof of Proposition \ref{p:discret}}
Consider the value $m_t:= \nu_t (\{x\dvtx\fp_{\nu_t \rightarrow
\mu_1}
(x) >0\})$. On one hand, the sequence $(m_t)$ converges to 0. Indeed,
$\nu_t$ is a part of the occupation measure of a random walk on $\Z$
with the initial distribution $\mu_0$, that is in particular
conditioned to never exit the interval $I:= \operatorname{ Supp} (\mu_1)$. The
probability of staying inside $I$ during $t$ steps converges to 0, and
thus, so does $m_t$. On the other hand, Lemma \ref{l:equality-discrete}
implies that
\[
\nu_t\bigl(\bigl\{x\dvtx\fp_{\nu_t \rightarrow\mu_1}(x)>0\bigr
\}\bigr) \ge
\mu_1 \bigl(\bigl\{x\dvtx\fp_{\nu_t \rightarrow\mu_1}(x)>0\bigr
\}\bigr)
\]
and thus
\[
m_t \ge\sharp\bigl\{x\dvtx\fp_{\nu_t \rightarrow\mu
_1}(x)>0\bigr\}\cdot\min
_{z\in I} \mu_1(z).
\]
As $b:= \min_{z\in I} \mu_1(z) >0$ due to the hypothesis of the
proposition, once $m_t <b$, we have $\fp_t \equiv0$ and hence $\nu_t =
\mu_1$.
\end{pf*}

\subsubsection{Proof of Theorem \texorpdfstring{\protect\ref{thm:BT-interval}}{2.4}}

We are now ready to prove Theorem \ref{thm:BT-interval}. Let two
centered measures $\mu_0$ and $\mu_1$, supported on some interval
$I\subset\R$, be given and assume that, for these measures, the
hypotheses (i)--(v) of the theorem are satisfied. Up to a rescaling of
space and time, we can assume that $I=[-1,1]$.

For any natural $n$, one can consider the discretized measures $\mu
_{0}^{(n)}$ and $\mu_{1}^{(n)}$ on $\frac{1}{n}\Z$, defined as
%
%
\begin{equation}
\label{eq:defi-disc} \mu_i^{(n)} \biggl(\frac{k}{n}
\biggr) = n \int_{{(k-1)}/{n}}^{
{(k+1)}/{n}} \biggl( 1- \biggl\llvert
x-\frac{k}{n} \biggr\rrvert\biggr) \,\mathrm{d}\mu_i (x),\qquad
i=0,1.
\end{equation}
Note that the measures $\mu_{0}^{(n)}$ and $\mu_{1}^{(n)}$ are
supported on the sets $\{-1, \frac{-n+1}{n},\break \ldots,  \frac{n-1}{n},
1\}
$, and have the same mean.

Consider now the corresponding random walks (with the elementary time
step~$\frac{1}{n^2}$) and the corresponding functions
\[
\fp_{\mu_i^{(n)}}^{({1}/{n})\Z} \biggl( \frac{k}{n} \biggr) =
\sum
_{y<{k}/{n}, y\in({1}/{n})\Z} \biggl( y-\frac{k}{n} \biggr
) \mu
_i^{(n)}(y),
\]
which, as earlier for $\Z$, are the restrictions on $\frac{1}{n}\Z$ of
the continuous functions $\fp_{\mu_i^{(n)}} (x)$. A first step in
applying the discretization technique\vspace*{-1.5pt} is a check that there exists a
discrete finite expectation Brownian transport from $\mu_{0}^{(n)}$ to
$\mu_{1}^{(n)}$.

%
\begin{lemma}
For any $n$ large enough, the measures $\mu_{0}^{(n)}$ and $\mu
_{1}^{(n)}$ satisfy the hypotheses of Proposition \ref{p:discret}.
\end{lemma}

\begin{pf}
Note that the functions $\fp_{\mu_{0}^{(n)} \to\mu_{1}^{(n)}}$
converge uniformly to the function $\fp_{\mu_0 \to\mu_1}$ that is
positive inside $I$. Hence,
\[
\forall\delta\ \exists n_0\dvtx\forall n>n_0\qquad
\fp_{\mu
_{0}^{(n)} \to\mu_{1}^{(n)}}|_{I\setminus\mathcal{U}_\delta
(\partial
I)} >0.
\]
On the other hand, due to the assumption (v), we have
\[
\exists n_1\dvtx\forall n\ge n_1, \forall x\in
\mathcal{U}_\varepsilon(\partial I) \cap I\cap\frac{1}{n}\Z
\qquad \mu_1^{(n)}(x) > \mu_0^{(n)}(x),
\]
what assures $\fp_{\mu_{0}^{(n)} \to\mu_{1}^{(n)}}|_{\mathcal
{U}_\varepsilon(\partial I) \cap I\cap({1}/{n})\Z} \ge0$.
Choosing then $\delta= \varepsilon/2$, we see that $\fp_{\mu
_{0}^{(n)} \to\mu_{1}^{(n)}}$ is positive everywhere on $I$ once $n$
is large enough.
\end{pf}

Consider now the corresponding discrete potentials $g_{(n)}(x)$ that we
extend to $[-1,1]$ piecewise linearly. Note that, for these functions,
we still have the (uniform in $n$) estimates, analogous to Lemma \ref
{l:decomp-I} and Proposition \ref{p:intersection} (proven by the same methods).
Hence, the family of functions $g_{(n)}$ is precompact and we can
extract a convergent subsequence $g_{(n_k)} \rightarrow f$. On the
other hand, discrete random walks tend, as $n\to\infty$, to the
Brownian motion. Hence, the same arguments as in Proposition \ref
{p:intersection} and Lemma \ref{l:conv} imply that the first
intersection measure for the initial distribution $\mu_0 = \lim_{k\to
\infty} \mu_0^{(n_k)}$ with the stopping function $f = \lim_{k\to
\infty
} g_{(n_k)}$ will be $\lim_{k\to\infty} \mu_1^{(n_k)} = \mu_1$. This
completes the proof of Theorem \ref{thm:BT-interval}.

\section*{Acknowledgements}
V.~Kleptsyn thanks S. Pirogov, A. Sobolevski, B. Gurevich, A. Rybko and S.
Shlosman for interesting discussions. A.~Kurtzmann thanks Bernard Roynette for
telling her the Cantelli conjecture. Both authors thank I.~Meilijson
and A. Cox for an interesting discussion.

%

%




\printaddresses

\end{document}